\documentclass[11pt]{article}
\usepackage{latexsym,amsfonts,amssymb,amsmath,amsthm}
\usepackage{graphicx}

\usepackage[usenames,dvipsnames]{color}
\usepackage{ulem}

\begin{document}
\setlength{\baselineskip}{16pt}

\parindent 0.5cm
\evensidemargin 0cm \oddsidemargin 0cm \topmargin 0cm \textheight 22.5cm \textwidth 16cm \footskip 2cm \headsep
0cm

\newtheorem{theorem}{Theorem}[section]
\newtheorem{lemma}{Lemma}[section]
\newtheorem{proposition}{Proposition}[section]
\newtheorem{definition}{Definition}[section]
\newtheorem{example}{Example}[section]
\newtheorem{corollary}{Corollary}[section]

\newtheorem{remark}{Remark}[section]

\numberwithin{equation}{section}

\def\p{\partial}
\def\I{\textit}
\def\R{\mathbb R}
\def\C{\mathbb C}
\def\u{\underline}
\def\l{\lambda}
\def\a{\alpha}
\def\O{\Omega}
\def\e{\epsilon}
\def\ls{\lambda^*}
\def\D{\displaystyle}
\def\wyx{ \frac{w(y,t)}{w(x,t)}}
\def\imp{\Rightarrow}
\def\tE{\tilde E}
\def\tX{\tilde X}
\def\tH{\tilde H}
\def\tu{\tilde u}
\def\d{\mathcal D}
\def\aa{\mathcal A}
\def\DH{\mathcal D(\tH)}
\def\bE{\bar E}
\def\bH{\bar H}
\def\M{\mathcal M}
\renewcommand{\labelenumi}{(\arabic{enumi})}

\def\disp{\displaystyle}
\def\undertex#1{$\underline{\hbox{#1}}$}
\def\card{\mathop{\hbox{card}}}
\def\sgn{\mathop{\hbox{sgn}}}
\def\exp{\mathop{\hbox{exp}}}
\def\OFP{(\Omega,{\cal F},\PP)}
\newcommand\JM{Mierczy\'nski}
\newcommand\RR{\ensuremath{\mathbb{R}}}
\newcommand\CC{\ensuremath{\mathbb{C}}}
\newcommand\QQ{\ensuremath{\mathbb{Q}}}
\newcommand\ZZ{\ensuremath{\mathbb{Z}}}
\newcommand\NN{\ensuremath{\mathbb{N}}}
\newcommand\PP{\ensuremath{\mathbb{P}}}
\newcommand\abs[1]{\ensuremath{\lvert#1\rvert}}

\newcommand\normf[1]{\ensuremath{\lVert#1\rVert_{f}}}
\newcommand\normfRb[1]{\ensuremath{\lVert#1\rVert_{f,R_b}}}
\newcommand\normfRbone[1]{\ensuremath{\lVert#1\rVert_{f, R_{b_1}}}}
\newcommand\normfRbtwo[1]{\ensuremath{\lVert#1\rVert_{f,R_{b_2}}}}
\newcommand\normtwo[1]{\ensuremath{\lVert#1\rVert_{2}}}
\newcommand\norminfty[1]{\ensuremath{\lVert#1\rVert_{\infty}}}
\newcommand{\ds}{\displaystyle}

\title{Chemotaxis systems with  singular sensitivity and logistic source: Boundedness,  persistence, absorbing  set, and  entire solutions}
\author{
Halil Ibrahim Kurt and Wenxian Shen   
\\
Department of Mathematics and Statistics\\
Auburn University\\
Auburn University, AL 36849\\
U.S.A. }

\date{}
\maketitle

\begin{abstract}
This paper deals with the following parabolic-elliptic chemotaxis system with
  singular sensitivity and logistic source,
\begin{equation}
\label{abstract-eq}
\begin{cases}
u_t=\Delta u-\chi\nabla\cdot (\frac{u}{v} \nabla v)+u(a(t,x)-b(t,x) u), & x\in \Omega,\cr
0=\Delta v- \mu v+ \nu u, & x\in \Omega, \cr
\frac{\p u}{\p n}=\frac{\p v}{\p n}=0, & x\in\p\Omega,
\end{cases}
\end{equation}
where $\Omega \subset \mathbb{R}^N$ is a smooth bounded domain,  $a(t,x)$ and $b(t,x)$ are positive smooth functions, and  $\chi$,  $\mu$ and $\nu$ are  positive constants. In recent years, a lot of attention has been drawn to the question of whether logistic kinetics  prevents  finite-time blow-up in various chemotaxis models. In the very recent paper \cite{HKWS}, we proved that for given nonnegative initial function $0\not\equiv u_0\in C^0(\bar \Omega)$ and $s\in\RR$, \eqref{abstract-eq} has a unique globally defined classical solution $(u(t,x;s,u_0),v(t,x;s,u_0))$
with $u(s,x;s,u_0)=u_0(x)$, provided that $a_{\inf}=\inf_{t\in\mathbb{R},x\in\Omega}a(t,x)$ is large relative to $\chi$ and $u_0$ is not small. 

In this paper, we  further investigate qualitative properties of globally defined positive solutions of \eqref{abstract-eq} under the assumption that $a_{\inf}$ is large relative to $\chi$ and $u_0$ is not small. Among others, we provide some concrete estimates for $\int_\Omega u^{-p}$ and $\int_\Omega u^q$ for some $p>0$ and $q>\max\{2,N\}$ and prove that    any globally defined positive solution is bounded  above and below eventually by some positive constants independent of its initial functions. We prove the existence of a ``rectangular'' type bounded  invariant set (in $L^q$) which eventually attracts all the globally defined positive solutions.
We also prove that
 \eqref{abstract-eq} has a positive entire classical solution $(u^*(t,x),v^*(t,x))$, which is periodic in $t$ if  $a(t,x)$ and $b(t,x)$ are periodic in $t$ and is independent of $t$   if $a(t,x)$ and $b(t,x)$ are independent of $t$.
\end{abstract}

\medskip


\vspace{4in}

\section{Introduction and Main Results}
\label{S:intro}

Chemotaxis systems, also known as Keller-Segel  systems, have been
widely studied since the pioneering works   \cite{Keller-0, Keller-00} by Keller and Segel at the beginning of 1970s
on the mathematical modeling of the aggregation process of Dictyostelium discoideum.
The current paper is devoted to the study of the asymptotic dynamics of the following  parabolic-elliptic chemotaxis system,
\begin{equation}
\label{main-eq}
\begin{cases}
u_t=\Delta u-\chi\nabla\cdot (\frac{u}{v} \nabla v)+u(a(t,x)-b(t,x) u),\quad &x\in \Omega,\cr
0=\Delta v-\mu v+\nu u,\quad &x\in \Omega, \quad \cr
\frac{\p u}{\p n}=\frac{\p v}{\p n}=0,\quad &x\in\p\Omega,
\end{cases}
\end{equation}
where  $\Omega\subset\RR^N$ is a smooth bounded domain, $a(t,x)$ and $b(t,x)$ are nonnegative smooth functions, and  $ \chi, \mu$ and $\nu$ are positive constants.
Biologically,  \eqref{main-eq}
 describes the evolution of  a biological process in which cells (with density $u$) move  towards higher concentrations of a chemical substance with density $v$ produced by cells themselves.
In \eqref{main-eq},  the cross-diffusion term  $-\chi\nabla\cdot (\frac{u}{v}\nabla v)$
reflects the chemotactic movement and $\frac{\chi}{v}$ is refereed to as chemotaxis sensitivity;  the reaction term $u(a(t,x)-b(t,x)u)$ represents the cell kinetic mechanism and is referred to as logistic source; $\mu>0$ denotes  the degradation rate of the chemical substance, and $\nu>0$ is the rate
at which the mobile species produces the chemical substance.  It  is seen that  the chemotaxis sensitivity $\frac{\chi}{v}$ is  singular near $v=0$,   reflecting  an inhibition of chemotactic migration at high signal concentrations. Such a sensitivity  was first proposed in \cite{Keller-1} due to the Weber–Fechner law of stimulus perception. The time and space dependence of the logistic source reflects the heterogeneity of   the underlying environment
of the chemotaxis system.

Since the pioneering works of Keller and Segel (\cite{Keller-0, Keller-00, Keller-1}), considerable efforts have been devoted to identifying the effects of the cross-diffusion and the kinetic term on the blow-up or global boundedness of solutions of \eqref{main-eq}.
For example, consider  \eqref{main-eq} without logistic source (i.e. $a(x,t) = b(x,t) \equiv 0$) and $\mu=\nu=1$. When $\Omega$ being a ball, it is shown in \cite{NaSe} that the classical radially symmetric positive solutions are global and bounded when $\chi>0$ and $N=2$,  or $\chi < \frac{2}{N-2}$ and $N \ge3$, and  there exist radial blow-up solutions if $\chi>\frac{2N}{N-2}$ and $N\ge 3$. Without the requirement for symmetry, Biler in \cite{Bil} proved the global existence of positive  solutions when $\chi\le 1$ and $N=2$, or $\chi<\frac{2}{N}$ and $N\ge 2$. Fujie,  Winkler, and Yokota in \cite{FuWiYo} proved the boundedness of globally defined
positive  solutions  when $\chi<\frac{2}{N}$ and $N\ge 2$.  More recently, Fujie and Senba in \cite{FuSe1} proved the global existence and boundedness of
classical positive solutions for the case of $N=2$ for any $\chi>0$. The existence of finite-time blow-up is then completely ruled out for any $\chi>0$  in the case $N=2$. In \cite{Bla}, global existence of weak solutions is proved if $0<\chi<\frac{N}{N-2}$.

Consider  \eqref{main-eq} with $a(t,x),b(t,x)>0$.
Central questions include  whether the logistic source prevents  the occurrence  of    finite-time blow-up
in \eqref{main-eq} (i.e. any positive solution exists globally);  if so,
whether the logistic source prevents  the occurrence  of    infinite-time blow-up
in \eqref{main-eq} (i.e.  any globally defined positive solution is bounded), and
what is the long time behavior of globally defined bounded positive solutions, etc.
To recall the existing  results  related to these central questions,   we first make the following standing assumption on $a(t,x)$ and $b(t,x)$:

\medskip

\noindent {\bf (H)} {\it $a(t,x)$ and $b(t,x)$ are  continuous in $x\in\bar\Omega$ uniformly
with respect to $t\in\RR$,
\text{ H\"older} continuous in $t\in\RR $ with exponent $0<\gamma_0<1$ uniformly with respect to $x\in\bar\Omega$, i.e, there is $B_1>0$ such that
$$
|a(t,x)-a(s,x)|\le B_1 |t-s|^{\gamma_0}, \quad |b(t,x)-b(s,x)|\le B_1 |t-s|^{\gamma_0}\quad \forall \, t,s\in\RR,\,\, x\in\bar\Omega,
$$
 and there are positive constants $\alpha$, $B_2$  such that }
$$ \alpha  \le a(t,x)\le B_2,\quad \alpha \le  b(t,x) \leq B_2.
$$
Put
\begin{equation*}
 \label{a-i-sup-inf-eq1}
a_{\inf}=\inf_{x\in\bar\Omega, t\in\RR}a(t,x),\,\, b_{\inf}=\inf _{ x \in\bar{\Omega},t\in\R} b(t,x),\,\, a_{\sup} =\sup_{x\in\bar\Omega,t\in\RR} a(t,x),\,\,  b_{\sup}=\sup _{x \in\bar{\Omega},t\in\R}b(t,x).
\end{equation*}
We consider the classical solutions $(u(t,x),v(t,x))$ of \eqref{main-eq} with initial functions $u_0(x)$
  satisfying
\begin{equation}
\label{initial-cond-eq}
u_0 \in C^0(\bar{\Omega}), \quad u_0 \ge 0, \quad {\rm and} \quad \int_\Omega u_0 >0.
\end{equation}

\begin{definition}
\label{solu-def}
For given $s\in\mathbb{R}$ and $u_0$ satisfying \eqref{initial-cond-eq}, we say $(u(t,x),v(t,x))$ is a {\rm classical solution} of \eqref{main-eq} on $(s,s+T)$ for some $T\in (0,\infty]$ with initial condition  $u(s,x)=u_0(x)$ if
\begin{equation*}
u(\cdot,\cdot)\in  C([s,s+T)\times\bar\Omega )\cap C^{1,2}(  (s,s+T)\times\bar\Omega),\quad
v(\cdot,\cdot)\in C^{0,2}((s,s+T)\times \bar\Omega),
\end{equation*}
\begin{equation}
\label{local-2-eq0}
\lim_{t\to s+}\|u(t,\cdot;s,u_0)-u_0(\cdot)\|_{C^0(\bar\Omega)}=0,
\end{equation}
and $(u(t,x),v(t,x))$ satisfies \eqref{main-eq} for all $(t,x)\in (s,s+T)\times \Omega$.
\end{definition}

Sometime,  we may assume
\begin{equation}
\label{main-assumption}
    a_{\inf}>
\begin{cases}
\frac{\mu \chi^2}{4}, &\text{if $0< \chi \leq 2,$}\\
\mu(\chi-1), &\text{if $\chi>2$.}\\
\end{cases}
\end{equation}

The  following proposition is on the existence and uniqueness
 of the classical solutions of \eqref{main-eq} with given initial function $u_0$ satisfying \eqref{initial-cond-eq} and follows from
 the arguments in \cite[Lemma 2.2]{FuWiYo1}.

\begin{proposition}
\label{local-existence-prop}(Local existence) {For any $s\in\mathbb{R}$ and $u_0$ satisfying \eqref{initial-cond-eq},  there is $T_{\max}(s,u_0)\in (0,\infty]$
such that the system \eqref{main-eq} possesses
 a unique  classical solution, denoted by $(u(t,x;s,u_0)$, $v(t,x;s,u_0))$,  on $(s,T_{\max}(s,u_0))$  with  initial condition $u(s,x;s,u_0)=u_0(x)$.
Furthermore,  if $T_{\max}(s,u_0)< \infty,$ then}
\begin{equation*}
\label{finite-time-blow-up}
\limsup_{t \nearrow T_{\max}(s,u_0)} \left\| u(t,\cdot;s,u_0) \right\|_{C^0(\bar \Omega)}  =\infty \quad or \quad \liminf_{t \nearrow T_{\max}(s,u_0)} \inf_{x \in \Omega} v(t,\cdot;s,u_0)   =0.
\end{equation*}
\end{proposition}

Note that, by the assumption {\bf (H)}, it is not difficult to prove that $\inf_{s\in \mathbb{R}} T_{\max}(s,u_0)>0$. 

We say that {\it finite-time blow-up} occurs in \eqref{main-eq} if for some $s\in\RR$ and $u_0$ satisfying \eqref{initial-cond-eq}, $T_{\max}(s,u_0)<\infty$, and {\it infinite-time blow-up} occurs if for some $s\in\RR$ and $u_0$ satisfying \eqref{initial-cond-eq}, $T_{\max}(s,u_0)=\infty$ and
\begin{equation*}
\label{infinite-time-blow-up}
\limsup_{t \to\infty} \left\| u(t,\cdot;s,u_0) \right\|_{C^0(\bar \Omega)}  =\infty \quad or \quad \liminf_{t \to\infty} \inf_{x \in \Omega} v(t,\cdot;s,u_0)   =0.
\end{equation*}

Throughout the rest of this section, we assume that $u_0(x)$ satisfies \eqref{initial-cond-eq} and
$(u(t,x),v(t,x)):=(u(t,x;s,u_0),v(t,x;s,u_0))$ is the unique classical solution of \eqref{main-eq} on  $(s,T_{\max}(s,u_0))$  with the initial
condition $u(s,x;s,u_0)=u_0(x)$.
We now  recall   some existing results related to the central questions mentioned in the above.

When $N=2$ and $a(t,x)\equiv a$, $b(t,x)\equiv b$ are positive constants, the following have been proved:

\begin{itemize}

\item[(a)] {\it For any $u_0$ satisfying \eqref{initial-cond-eq} and $s\in\RR$,  $T_{\max}(s, u_0)=\infty$, that is, finite-time blow-up does not occur}  (see \cite[Theorem 1.1]{FuWiYo1}).

\item[(b)] {\it  If \eqref{main-assumption} holds, then  for any $u_0$ satisfying \eqref{initial-cond-eq} and $s\in\RR$,
 $\sup_{t\ge s}\|u(t,\cdot;s,u_0)\|_\infty<\infty$,  that is, infinite-time blow-up does not occur under the assumption \eqref{main-assumption}}
(\cite[Theorem 1.2]{FuWiYo1}).

\item[(c)] {\it  The constant solution
$(\frac{a}{b},\frac{\nu}{\mu}\frac{a}{b})$  is exponentially stable under the assumption \eqref{main-assumption} and some other assumptions}
(see \cite[Theorem 1]{CaWaYu}).
\end{itemize}

The authors of the current paper studied the global existence of  classical solutions of \eqref{main-eq} for  general $N\ge 1$ and time and space dependent functions $a(t,x)$, $b(t,x)$ in \cite{HKWS}. To recall some results proved in \cite{HKWS}, we first state some additional  conditions on  initial data $u_0$ and on the coefficients in \eqref{main-eq}.
Here is an additional condition on the initial function $u_0$:
$$
u_0\,\, \text{satisfies}\,\, \eqref{initial-cond-eq}\,\,\, {\rm and}\,\, \,
 \exists\, \tau_0>0\,\, \,{\rm s.t.}\,\,\, \int_\Omega u^{-1}(s+\tau_0,x;s,u_0)\le \frac{b_{\sup}|\Omega|\max\{1,\frac{1}{\chi}\}}{a_{\inf}-a_{\chi,\mu}}\,\,\, \forall\, s\in\mathbb{R}.\eqno(1.2)'
$$
By \cite[Proposition 1.3]{HKWS} (see also Lemma \ref{main-lem2}),
for any $q \ge  3$ and $q-1 \leq k <  2q-2,$ there exist positive constants $M(k,q)>0$ and $M^{*}(k,q)>0$ such that   for any given $s\in\mathbb{R}$ and  $u_0$ satisfying \eqref{initial-cond-eq},
\begin{equation*}
   \int_{\Omega} \frac{|\nabla v(x,t;s,u_0)|^{2q}}{v^{k+1}(x,t;s,u_0)} \leq M(k,q) \int_{\Omega} \frac{u^q(x,t;s,u_0)}{v^{k-q+1}(x,t;s,u_0)}+M^{*}(k,q)\int_{\Omega} v^{2q-k-1}(x,t;s,u_0)
\end{equation*}
for all  $t \in (s, T_{\rm max}(s,u_0))$. The following is an additional condition on the parameters:
$$
a_{\inf}> a_{\chi,\mu}+ \frac{b_{\sup}|\Omega| (p_N-1) \big(C_n^*\big)^{\frac{1}{p_N+1}}\max\{\chi,\chi^2\}}{4b_{\inf} \delta_0},\eqno(1.5)'
$$
where  $\delta_0$ is as in Lemma \ref{prelim-lm-00},  $p_N=\max\{2,N\}$,  $C_N^*=M(p_N,p_N+1)$,  and
\begin{equation*}
   a_{\chi,\mu}:=2\big(\chi+2-2\sqrt{\chi+1}\big)\mu.
\end{equation*}

Note that if $\int_\Omega u_0^{-1}(x)dx\le \frac{b_{\sup}|\Omega|\max\{1,\frac{1}{\chi}\}}{a_{\inf}-a_{\chi,\mu}}$, then $(1.2)'$ holds.
The condition $(1.2)'$ indicates that $u_0$ is not  small, which prevents $v$ becomes too small as time evolutes and is a natural assumption. Note  also that $a_{\chi,\mu}$ in $(1.5)'$ satisfies
\begin{equation*}
    a_{\chi,\mu}\le \begin{cases} \frac{\mu \chi^2}{2},\quad &{\rm if}\,\,\,\,  0<\chi\le 2\cr
2\mu(\chi-1), \quad &{\rm if}\,\,\, \,  \chi>2.
\end{cases}
\end{equation*}
 The condition $(1.5)'$ indicates that $a(\cdot,\cdot)$ is large relative to the chemotaxis sensitivity coefficient $\chi$, which also prevents $v$ becomes too small as time evolutes, and is a natural condition. Assuming that $(1.2)'$ and $(1.5)'$ hold, among others, the following  are proved in the recent paper \cite{HKWS}:

\begin{itemize}
\item[(i)] (Global existence)  {\it  For any $s\in\RR$},
 \begin{equation}
\label{paper1-eq1}
T_{\max}(s,u_0)=\infty
\end{equation}
 (see \cite[Theorem 1.2(3)]{HKWS}).

\item[(ii)] (Boundedness of $\int_\Omega u^q$) {\it There is $q>N$ such that for any  $s\in\RR$,}
\begin{equation}
\label{paper1-eq3}
\sup_{t\ge s}\int_\Omega u^q(t,x;s,u_0)dx<\infty
\end{equation}
(see \cite[Theorem 1.1(3)]{HKWS}).

\item[(iii)] (Boundedness of $\int_\Omega u^{-p}$)  {\it There is $p>0$ such that for any  $s\in\RR$}, 
\begin{equation}
\label{paper1-eq4}
\limsup_{t\to\infty}\int_\Omega u^{-p}(t,x;s,u_0)dx<\infty
\end{equation}
(see \cite[Lemma 3.4]{HKWS}).

\item[(iv)] (Mass persistence) {\it For any  $s\in\RR$},
 \begin{equation}
\label{paper1-eq5}
\inf_{t\ge s}\int_\Omega u(t,x;s,u_0)dx>0,\quad \inf_{t\ge s,x\in\Omega} v(t,x;s,u_0)>0
\end{equation}
(see \cite[Proposition 1.2(2)]{HKWS}).

\item[(v)] (Boundedness of $\|u\|_\infty$) {\it For any   $s\in\RR$,}
\begin{equation}
\label{paper1-eq2}
\sup_{t\ge s}\|u(t,\cdot;s,u_0)\|_\infty<\infty
\end{equation}
 (see \cite[Theorem 1.2(3)]{HKWS}).
\end{itemize}

We remark  that \eqref{paper1-eq2} implies \eqref{paper1-eq3}. But \eqref{paper1-eq2} is proved in \cite{HKWS} by using \eqref{paper1-eq3} and \eqref{paper1-eq5},  and \eqref{paper1-eq5} is proved in \cite{HKWS} by using \eqref{paper1-eq4}.
Hence the boundedness of $\int_\Omega u^q$ in \eqref{paper1-eq3} and the boundedness
of $\int_\Omega u^{-p}$ in \eqref{paper1-eq4} play essential roles
in the proofs of the main results in \cite{HKWS}.

The results (a)-(c) and (i)-(v) provide some answers to the central questions mentioned in the above. For example, results (i)-(v) imply that logistic
kinetics prevents the occurrence of  finite-time blow-up provided that the intrinsic growth rate function $a(t,x)$ is large relative to the chemotaxis sensitivity and the initial condition
$u_0$ is not small, which is an interesting biological phenomenon.
However, there are still many interesting questions associated to those central questions. For example,
 whether the ultimate upper bound of $\int_\Omega u^q$, $\int_\Omega u^{-p}$,  and $\|u\|_\infty$ in (ii), (iii), and
(v), respectively, are independent of the initial functions; whether the {\it  mass persistence} in (iv) (i.e.
$\int_{t\ge s}\int_\Omega u(t,x;s,u_0)dx>0$) can be replaced by the {\it pointwise persistence}  (i.e.
$\liminf_{t-s\to\infty}\inf_{x\in\Omega} u(t,x;s,u_0)>0$);  whether \eqref{main-eq} has {\it bounded positive  entire solutions} (i.e. bounded positive solutions $(u(t,x),v(t,x))$ which are defined for all  $t\in\mathbb{R}$); whether (i)-(v) hold without the assumptions $(1.2)'$ and $(1.5)'$.

In the current paper, we will further investigate those central questions mentioned in the above.
We state our main results in next subsection.

\subsection{Main results and remarks}

Assume that $(1.5)'$ holds.
For any $s\in\mathbb{R}$, $\tau\ge s$, $u_0$ satisfying $(1.2)'$, and $p>0$,  let
\begin{equation*}
 m^*(\tau,s,u_0)= {\rm max}\Big\{\int_{\Omega} u(\tau,x;s,u_0)dx,  \frac{a_{\sup}}{b_{\inf}}|\Omega| \Big\}
\end{equation*}
and
\begin{equation*}
\tilde M_1(p,\tau,s,u_0)= p b_{\sup} |1-p| \Big( m^*(\tau,s,u_0)-\frac{a_{\sup}}{b_{\inf}}|\Omega|\Big).
\end{equation*}
Note that
$$
u(t,x;s,u_0), v(t,x;s,u_0)>0\quad \forall\, x\in\Omega,\,\, t\in (s,\infty).
$$
A {\it positive entire solution} of \eqref{main-eq} is a solution $(u(t,x),v(t,x))$ which is defined for all $t\in\mathbb{R}$ and
$\inf_{t\in\mathbb{R},x\in\Omega}u(t,x)>0$.

The first main result of the current paper  is on   the  boundedness of
$\int_\Omega u^{-p}$ and  $\int_\Omega u^q$ of the globally defined classical solutions of \eqref{main-eq}.

\begin{theorem} [Boundedness   of
$\int_\Omega u^{-p}$ and $\int_\Omega u^q$]
\label{main-thm1}
Assume that $(1.5)'$  holds.  Then the following hold.
\begin{itemize}
\item[(1)]  Let $p=1$, $\gamma=a_{\inf}-a_{\chi,\mu}$,  $M_1(p)=\frac{b_{\sup}|\Omega|}{a_{\inf}-a_{\chi,\mu}}$. Then for any $u_0$ satisfying $(1.2)'$, $s\in \mathbb{R}$ and  $\tau>s$,
\end{itemize}
\begin{align*}
    \int_\Omega u^{-p}(t,x;s,u_0)dx\le e^{-\gamma (t-\tau)} \int_\Omega u^{-p}(\tau,x;s,u_0)dx  + M_1(p)+\tilde M_1(p,u_0,\tau)\quad \forall\, t \ge \tau,
\end{align*}
and
\begin{equation}
\label{revised-negative-p-eq1}
\int_\Omega u^{-p}(t,x;s,u_0)dx\le \max\left\{\int_\Omega u^{-p}(\tau,x;s,u_0)dx, M_1(p)+\tilde M_1(p,\tau,s,u_0)\right\}\,\,\,\forall\,\, t\ge \tau,
\end{equation}
as well as
\begin{equation}
\label{revised-negative-p-eq2}
\lim_{t-s\to\infty} \int_\Omega u^{-p}(t,x;s,u_0)\le M_1(p).
\end{equation}

\smallskip

\begin{itemize}
\item[(2)] Let $p>0$ be as in (1). There are $q>\max\{2,N\}$ and  $M_2(p,q)>0$  such that for any $s\in\mathbb{R}$, $u_0$ satisfying $(1.2)'$, and $\tau>s$, there is $\tilde M_2(p,q,\tau,s,u_0)$ such that
\end{itemize}
\begin{equation*}
 \int_\Omega u^q(t,x;s,u_0)dx\le  e^{-(t-\tau)}\int_\Omega u^q(\tau,x;s,u_0)dx +M_2(p,q)+\tilde M_2(p,q,\tau,s,u_0)\quad \forall\, t\ge \tau,
\end{equation*}
and
\begin{equation*}
 \int_\Omega u^q(t,x;s,u_0)dx\le  \max\left\{\int_\Omega u^q(\tau,x;s,u_0)dx, M_2(p,q)+\tilde M_2(p,q,\tau,s,u_0)\right\}\quad \forall\, t\ge \tau,
\end{equation*}
as well as
\begin{equation*}
\limsup_{t-s\to\infty}\int_\Omega u^q(t,x;s,u_0)ds\le M_2(p,q).
\end{equation*}
\end{theorem}

\smallskip

\begin{remark}
\label{rmrk-1}
\begin{itemize}

\item[(1)]  Theorem \ref{main-thm1}(1)  improves boundedness result of $\int_\Omega u^{-p}(t,x;s,u_0)dx$ in   \cite[Lemma 3.4]{HKWS} (see (iii) in the above)  in the following two aspects. First, it provides some  concrete estimates for $\int_\Omega u^{-p}(t,x;s,u_0)dx$. Second, it provides an ultimate upper bound  independent of $u_0$
for $\int_\Omega u^{-p} (t,x;s,u_0)dx$.

\item[(2)]
Theorem \ref{main-thm1}(2)  improves the boundedness  result of $\int_\Omega u^q(t,x;s,u_0)dx$ in    \cite[Theorem 1.1(2)]{HKWS} (see (ii) in the above)  in the following two aspects. First, it provides some concrete estimates for $\int_\Omega u^{q}(t,x;s,u_0)dx$. Second, it provides an ultimate upper bound  independent of $u_0$
for $\int_\Omega u^{q} (t,x;s,u_0)dx$.

\item[(3)] The inequality \eqref{revised-negative-p-eq2} provides some improvement of  \eqref{paper1-eq5} proved in \cite{HKWS}. To be more precise, let $p$ be as in Theorem \ref{main-thm1}(1). By the H\"older's  inequality, for any $s\in\RR$ and $u_0$ satisfying $(1.2)'$, we have
\begin{align*}
|\Omega|&=\int_\Omega u^{\frac{p}{p+1}}(t,x;s,u_0) u^{-\frac{p}{p+1}}(t,x;s,u_0)dx\nonumber\\
&\le \left(\int_\Omega  u(t,x;s,u_0)dx\right)^{\frac{p}{p+1}} \left(\int_\Omega u^{-p}(t,x;s,u_0)dx \right)^{\frac{1}{p+1}}\quad \forall\, t>s.
\end{align*}
Then by \eqref{revised-negative-p-eq2},
\begin{equation}
\label{new-set-eq4}
\lim_{t-s\to\infty}\int_\Omega u(t,x;s,u_0)dx \ge \frac{|\Omega|^{\frac{p+1}{p}}}{(2M_1(p))^{\frac{1}{p}}},
\end{equation}
which provides an ultimate lower bound independent of $u_0$ for $\int_\Omega u(t,x;s,u_0)dx$ and hence  improves \eqref{paper1-eq5} proved in \cite{HKWS}.

\item[(4)]  The upper bounds of $\int_\Omega u^{-p}(t,x;s,u_0)dx$ and
$\int_\Omega u^q(t,x;s,u_0)dx$ obtained in Theorem \ref{main-thm1} provide some useful  tool for the study of the asymptotic behavior of globally defined positive solutions of \eqref{main-eq} (see Theorems \ref{main-thm2}-\ref{main-thm4} in the following).
\end{itemize}
\end{remark}

The second main result of this paper is on the existence of globally absorbing  sets, which provides some insight on the asymptotic behavior of globally defined positive solutions of \eqref{main-eq}.

\begin{theorem} [Globally absorbing rectangle]
\label{main-thm2}
Assume that $(1.5)'$ holds.  Let $q>\max\{2,N\}$ and $p>0$ be as in Theorem \ref{main-thm1}. There are $M_0^*>0$, $M_1^*,M_2^*>0$ such that the following hold.
\begin{itemize}
\item[(1)] The set
\begin{equation}
\label{set-E-eq}
\mathcal{E}=\left\{u\in C^0(\bar \Omega)\,|\, u\ge 0,\,\, { \int_\Omega u(x)dx\le M_0^*},  \int_\Omega u^{-p}(x)dx\le M_1^*,\,\, \int_\Omega u^{q}(x)dx\le M_2^*\right\}
\end{equation}
is an  invariant set of \eqref{main-eq} in the sense that  for any $u_0\in\mathcal{E}$ and $s\in\mathbb{R}$,
$u_0$ satisfies $(1.2)'$ and
$u(t,\cdot;s,u_0)\in\mathcal{E}$ for all $t\ge s$. Moreover,  for any $0<\theta<1-\frac{2N}{q}$ and $\tau>0$, there is $M_3^*(\theta,\tau)>0$ such that for any $u_0\in\mathcal{E}$ and $s\in\RR$,
\begin{equation}
\label{new-set-eq1}
\|u(t,\cdot;s,u_0)\|_{C^\theta(\bar\Omega)}\le M_3^*(\theta,\tau)\quad \forall\, t\ge s+\tau.
\end{equation}

\item[(2)] The set $\mathcal{E}$  is globally absorbing in the sense that  for any $u_0$  satisfying  $(1.2)'$  and $s\in\mathbb{R}$,
\begin{equation}
\label{new-set-eq2}
\begin{cases}
\displaystyle \limsup_{t-s\to\infty}\int_\Omega u(t,x;s,u_0)\le M_0^*\cr
\displaystyle \limsup_{t-s\to\infty} \int_\Omega u^{-p}(t,x;s,u_0)dx\le M_1^*\cr
\displaystyle \limsup_{t-s\to\infty}\int_\Omega u^{q} (t,x;s,u_0)dx\le M_2^*.
\end{cases}
\end{equation}
Moreover,   for any $0<\theta<1-\frac{2N}{q}$, there is $M_4^*(\theta)>0$ such that for any $u_0$ satisfying $(1.2)'$, there is $T(u_0)>0$ such that
\begin{equation}
\label{new-set-eq3}
\|u(t,\cdot;s,u_0)\|_{C^\theta(\bar\Omega)}\le M_4^*(\theta)\quad \forall\, s\in\RR,\,\,  t\ge s +T(u_0).
\end{equation}
\end{itemize}
\end{theorem}

\begin{remark}
\label{rmrk-2}
\begin{itemize}

\item[(1)] Assume that $(1.5)'$  holds. \eqref{new-set-eq3} provides some improvement
of \eqref{paper1-eq2}
proved in \cite{HKWS}. In fact,   by \eqref{new-set-eq3}, for any $s\in\RR$ and $u_0$ satisfying $(1.2)'$,
\begin{equation}
\label{new-set-eq5}
\|u(t,\cdot;s,u_0)\|_\infty \le M_4^*(\theta)\quad \forall t-s\gg 1,
\end{equation}
which provides an ultimate upper bound independent of $u_0$ for $\|u(t,\cdot;s,u_0)\|_\infty$ and hence improves \eqref{paper1-eq2}
proved in \cite{HKWS}.

\item[(2)]  Theorem \ref{main-thm2}(2) implies that $\mathcal{E}$ is both pullback absorbing and forward absorbing  in the sense that for any  $u_0$ satisfying $(1.2)'$,
\begin{equation*}
\begin{cases}
\displaystyle \limsup_{s\to -\infty}\int_\Omega u(t,x;s,u_0)dx\le M_0^*\quad &\forall\, \, t\in\mathbb{R}\cr
\displaystyle \limsup_{s\to -\infty} \int_\Omega u^{-p}(t,x;s,u_0)dx\le M_1^*\quad &\forall\,\, t\in\mathbb{R}\cr
\displaystyle \limsup_{s\to- \infty}\int_\Omega u^{q} (t,x;s,u_0)dx\le M_2^*\quad &\forall\, \, t\in\mathbb{R},
\end{cases}
\end{equation*}
and
\begin{equation*}
\begin{cases}
\displaystyle \limsup_{t\to\infty}\int_\Omega u(t,x;s,u_0)\le M_0^*\quad &\forall\,\, s\in\mathbb{R}\cr
\displaystyle \limsup_{t\to \infty} \int_\Omega u^{-p}(t,x;s,u_0)dx\le M_1^*\quad &\forall\,\, s\in\mathbb{R}\cr
\displaystyle \limsup_{t\to \infty}\int_\Omega u^{q} (t,x;s,u_0)dx\le M_2^*\quad &\forall\,\, s\in\mathbb{R},
\end{cases}
\end{equation*}
 respectively.

\item[(3)] Note that the existence of bounded absorbing sets and eventual initial-independent boundedness of positive solutions are strongly related.
It should be pointed out that the existence of absorbing sets and eventual  boundedness
of positive solutions  have been studied in various logistic chemotaxis models. For example,  in \cite{Win-new}, Winkler
studied the existence of  bounded absorbing sets in $L^\infty$  for the following parabolic-elliptic chemotaxis model with regular sensitivity and logistic type source,
$$
\begin{cases}
u_t=\Delta u-\chi\nabla\cdot (u\nabla v)+g(u),\quad &x\in\Omega\cr
0=\Delta v-v+u,\quad &x\in\Omega\cr
\frac{\p u}{\p n}=\frac{\p v}{\p n}=0,\quad &x\in\p\Omega,
\end{cases}
$$
where $g(u)=Au-Bu^\alpha$ for some $A,B>0$ and $\alpha>1$.
We refer the reader to \cite{Vig-new1}, \cite{Vig-new2} and references therein for the study of eventual boundedness of positive solutions of the following parabolic-parabolic chemotaxis model,
$$
\begin{cases}
u_t=\Delta u-\chi\nabla\cdot (u\nabla v)+g(u),\quad &x\in\Omega\cr
v_t=\Delta v-v+u,\quad &x\in\Omega\cr
\frac{\p u}{\p n}=\frac{\p v}{\p n}=0,\quad &x\in\p\Omega,
\end{cases}
$$
where $g(u)$ is a logistic type nonlinear function.
\end{itemize}
\end{remark}

The following theorem is on the  pointwise persistence, which is strongly related to \eqref{new-set-eq4}  and   provides some further insight on  the asymptotic behavior of globally defined positive solutions of \eqref{main-eq}.

\medskip

\begin{theorem} [Uniform pointwise persistence]
\label{main-thm3}
Assume that $(1.5)'$ holds. There is $m^*>0$ such that for any $u_0$ satisfying
$(1.2)'$ and $s\in\RR$,
\begin{equation}
\label{new-set-eq6}
 \liminf_{t-s\to\infty} \inf_{x\in\Omega}u(t,x;s,u_0)\ge m^*.
\end{equation}
\end{theorem}

\begin{remark}
\label{rmrk-3}
The pointwise persistence \eqref{new-set-eq6} is proved for the first time for chemotaxis models with singular sensitivity and logistic source,  and implies the mass persistence.  We point out that mass persistence  for chemotaxis systems with regular chemotaxis sensitivity and logistic source  was studied in \cite{TaWi} and pointwise persistence  was studied in  \cite{IsSh1, IsSh2}.
\end{remark}

By \eqref{new-set-eq5} and \eqref{new-set-eq6}, under the assumption $(1.5)'$, any globally defined positive solution with initial condition $u_0$ satisfying $(1.2)'$  is eventually bounded above and below by some positive constants independent of its initial condition. It is natural to ask whether globally  defined positive solutions converge to some positive entire solution.
  When $a(t,x)\equiv a$ and $b(t,x)\equiv b$,  it is clear that $(\frac{a}{b},\frac{\nu}{\mu}\frac{a}{b})$ is a positive  entire solution of \eqref{main-eq}.
When $a(t,x)$ and $b(t,x)$ depend on $t$ and $x$, it is not clear at all whether \eqref{main-eq} has positive entire solutions.
The last main result of the current paper is on the existence of positive entire solutions of \eqref{main-eq}.

\begin{theorem} [Existence of positive entire solutions]
\label{main-thm4}
Suppose that $(1.5)'$  holds. Let $\mathcal{E}$ be as in Theorem \ref{main-thm2}.  Then the followings hold.
\begin{itemize}
\item[(1)] There is a positive entire solution $(u^*(t,x),v^*(t,x))$ of \eqref{main-eq} with
$u^*(t,\cdot)\in\mathcal{E}$ for any $t\in\RR$.

\item[(2)]
 If there is $T>0$ such that $a(t+T,x)=a(t,x)$ and $b(t+T,x)=b(t,x)$, then \eqref{main-eq} has a positive $T$-periodic solution $(u,v)=(u^*(t,x),v^*(t,x))$  with  $u^*(t,\cdot)\in\mathcal{E}$ for any $t\in\RR$.

\item[(3)]
  If $a(t,x)\equiv a(x)$ and $b(t,x)\equiv b(x)$, then \eqref{main-eq} has a positive stationary solution
  $(u,v)=(u^*(x),v^*(x))$ with $u^*(\cdot)\in\mathcal{E}$.
\end{itemize}
\end{theorem}

\begin{remark}
\label{rmrk-4}
As stated in (c), when $N=2$ and $a(t,x)\equiv a$, $b(t,x)\equiv b$, the constant entire solution
$(\frac{a}{b},\frac{\nu}{\mu}\frac{a}{b})$ is exponentially stable under the assumption \eqref{main-assumption}
and some other assumption (see \cite[Theorem 1]{CaWaYu}). By the  arguments of \cite[Theorem 1]{CaWaYu},
it can also be proved that,  for general $N\ge 1$,   $(\frac{a}{b},\frac{\nu}{\mu}\frac{a}{b})$ is exponentially stable under the assumption \eqref{main-assumption}
and some other assumption. But when $a(t,x), b(t,x)$ are not constant functions, the arguments of \cite[Theorem 1]{CaWaYu}
are difficult to apply. We plan to study the stability of positive entire solutions of \eqref{main-eq} somewhere else.
\end{remark}

It should be pointed out that  a considerable amount of research has also been carried out toward  the finite-time  blow-up or global boundedness of solutions of the following parabolic-parabolic
chemotaxis model with singular sensitivity,
\begin{equation}
\label{parabolic-parabolic-eq}
\begin{cases}
u_t=\Delta u-\chi\nabla\cdot (\frac{u}{v} \nabla v)+u(a-bu),\quad &x\in \Omega,\cr
v_t=\Delta v - v+ u,\quad &x\in \Omega, \quad \cr
\frac{\p u}{\p n}=\frac{\p v}{\p n}=0,\quad &x\in\p\Omega.
\end{cases}
\end{equation}
For example, for the case that $a=b=0$, it is proved  that if $0<\chi<\sqrt{\frac{2}{N}}$,
then for any initial data $u_0\in C^0(\bar\Omega)$ and $v_0\in W^{1,\infty}(\Omega)$ with
$u_0\ge 0$ and $v_0>0$ in $\bar\Omega$, there exists a
 global-in-time classical solution of \eqref{parabolic-parabolic-eq} (see \cite{Win4}),  and moreover, the global-in-time classical solution is bounded (see \cite{Fuj}). When $\Omega$ is a smooth, bounded, convex two-dimensional domain, it is shown in \cite{Lan} that there is $\chi_0>1$ such that  for any initial data $u_0\in C^0(\bar\Omega)$ and $v_0\in W^{1,\infty}(\Omega)$ with
$u_0\ge 0$ and $v_0>0$ in $\bar\Omega$, \eqref{parabolic-parabolic-eq} has a  global bounded solution for $\chi\in (0,\chi_0)$, which implies that $0<\chi<1$ is not critical for the global existence of \eqref{parabolic-parabolic-eq} on two-dimensional domains. See \cite{FuSe0, LaWi, Win4} for more results.

For the case that $a,b>0$,  it is proved that if $N=2$, then for any initial data $0\le u_0\in L^2(\Omega)$ and
$0<v_0\in H^{1+\epsilon_0}(\Omega)$ with $\inf_{x\in\Omega}v_0(x)>0$ and $\epsilon_0\in (0,\frac{1}{2})$, \eqref{parabolic-parabolic-eq} has a global classical solution (see  \cite{Aid, ZhZh}), and if \eqref{main-assumption} holds with $\mu=1$ and $a_{\inf}=a$, then globally defined positive solutions of \eqref{parabolic-parabolic-eq} are bounded (see \cite{ZhZh}) and the constant solution $(\frac{a}{b},\frac{a}{b})$ is exponentially stable (see \cite[Theorem 1.1]{ZhMuWiHu}).  See \cite{ZhZh1} for the existence of weak solutions of \eqref{parabolic-parabolic-eq} with $a,b>0$ in the case $N\ge 3$. However, there is little study on the global existence of classical solutions of \eqref{parabolic-parabolic-eq} when $N\ge  3$.

We also point out that a considerable amount of research has  been carried out toward  the finite-time blow-up or global boundedness of solutions of
 the following chemotaxis model with regular chemotaxis sensitivity and logistic source,
  \begin{equation}
\label{main-eq0}
\begin{cases}
u_t=\Delta u- \chi\nabla \cdot (u\nabla v)+u(a-bu),\quad &x\in \Omega, \cr
\tau v_t=\Delta v-\mu v+\nu u,\quad &x\in \Omega,\cr
\frac{\p u}{\p n}=\frac{\p v}{\p n}=0,\quad &x\in \partial\Omega.
\end{cases}
\end{equation}
 For example,  when $\tau=0$ and  $\mu=\nu=1$,  it is proved in \cite{TeWi1} that,  if $N\le 2$ or  $b>\frac{N-2}{N}\chi$, then
for every nonnegative initial data $u_0\in C^0(\bar\Omega)$, \eqref{main-eq0}  possesses a global bounded classical solution which is unique.
It should be pointed out that, when $a=b=0$ and $N\ge 2$,
finite-time blow-up of positive  solutions occurs under some condition on the mass and the
moment of the initial data   (see \cite{HeMeVe}, \cite{HeVe},  \cite{Nag2}, \cite{NaSe3}).
Hence the finite time blow-up phenomena in \eqref{main-eq0} is
suppressed to some extent by  the logistic source.
But it remains open whether in any space dimensional setting, for every nonnegative initial data  $u_0\in C^0(\bar\Omega)$ \eqref{main-eq0} possesses a unique global  classical solution  for every $\chi>0$ and every $b>0$. {It should be  pointed out that finite-time blow-up  occurs in  various variants of \eqref{main-eq0}, for example, it occurs in \eqref{main-eq0} with $\tau=0$, with  the logistic source being replaced by logistic-type superlinear degradation (see \cite{TaYo, Win5}), and/or with the second equation being replaced by the following one
  $$
  0=\Delta v-\frac{1}{|\Omega|}\int_\Omega u(\cdot,t)+u,\quad x\in\Omega
  $$
  (see \cite{BlFuLa, Fue1, Fue2, Win2}).}

The rest of this paper is organized in the following way. In section 2, we present some lemmas to be used in later sections. In section 3, we study the boundedness of
$\int_\Omega u^{-p}$ and $\int_\Omega u^q$  of  globally defined positive solutions of \eqref{main-eq} and prove the Theorem \ref{main-thm1}.  In section 4, we investigate
the existence of absorbing invariant sets and  the ultimate upper and lower bounds of globally defined positive solutions and
prove Theorems \ref{main-thm2} and \ref{main-thm3}. In section 5,
we explore the existence of   positive entire solutions of \eqref{main-eq} and  prove the Theorem \ref{main-thm4}.

\section{Preliminaries}

In this section, we present some lemmas to be used in later sections. Throughout this section, $\mu>0$ and $\nu>0$ are fixed.
 $C$ denotes some generic constant which is independent of solutions, but may not be the same at different places.

First, for convenience, we present two lemmas on fixed point theorems.

\begin{lemma}
\label{cont-family}
If $F$ is a closed subset of a Banach space $X,$ $G$ is a subset
of a Banach space $Y,$ $T_y:F \to F,$ $y$ in $G$ is a uniform contraction on $F$ and $T_y x$ is continuous in $y\in G$ for each fixed $x$ in $F,$ then the unique fixed point $g(y)$ of $T_y,$ $y$ in $G,$ is continuous in $y$.
\end{lemma}

\begin{proof}
See \cite[Theorem 3.2 in Chapter 0]{hale}.
\end{proof}

\begin{lemma}
\label{sch-fxd-point}
Let $G$ be a closed convex set in a Banach space $X$ and let $T$ be a continuous mapping of $G$ into itself such that the image $TG$ is precompact. Then $T$ has a fixed point.
\end{lemma}

\begin{proof}
It follows from \cite[Corollary 11.2]{gitr}.
\end{proof}

Next, we present two lemmas  on the lower and upper bounds of  the solutions of
 \begin{equation}
\label{v-eq1}
    \begin{cases}
    \Delta v-\mu v+\nu u=0, \quad &x\in \Omega, \cr
    \frac{\p v}{\p n}=0, \quad \quad \quad &x\in \p \Omega.
    \end{cases}
\end{equation}
For given $u\in L^p(\Omega)$, let $v(\cdot;u)$ be the solution of \eqref{v-eq1}.

\begin{lemma}
\label{prelim-lm-00}
 Suppose that $u \in C^0(\bar \Omega)$ is nonnegative  and  $\int_{\Omega} u >0$.
Then
\begin{equation*}
    v(x;u) \ge {\delta_0  }\int_{\Omega} u>0 \quad {\rm in} \quad \Omega
\end{equation*}
for some positive constants $\delta_0$   depending only on $\Omega$.
\end{lemma}

\begin{proof}
It follows from the arguments of  \cite[Lemma 2.1]{FuWiYo}.
\end{proof}

\begin{lemma}
\label{prelim-lm-000}
For any $p> 1$, there exists $C_p>0$ such that
\begin{equation*}
 \max\Big\{\|v(\cdot;u)\|_{L^p(\Omega)},   \| \nabla v(\cdot;u) \|_{L^p(\Omega)}\Big\} \leq C_p \|u(\cdot)\|_{L^p(\Omega)}\quad \forall u\in L^p(\Omega).
\end{equation*}
\end{lemma}

\begin{proof}
It follows from $L^p$-estimates for elliptic equations (see \cite[Theorem 12.1]{Ama}).
\end{proof}

Now, we present some properties of the semigroup generated by $-\Delta+\mu I$ complemented with Neumann boundary condition on $L^p(\Omega)$.
For given ${1< p<\infty}$,  let ${ X_p}=L^{p}(\Omega)$
  and ${ A_p}=-\Delta+\mu I$ with
\begin{equation*}
    D(A_p)=\left\{ u \in W^{2,p}(\Omega) \, |  \, \frac{\p u}{\p n}=0 \quad \text{on } \, \p \Omega \right\}.
\end{equation*}
It is well known that  $A_p$ is a sectorial operator in $X_p$  and thus generates an analytic semigroup $\left(e^{-A_pt}\right)_{t\geq 0}$ in $X_p$ (see, for example, \cite[Theorem 13.4]{Ama}). Moreover $0 \in \rho(A_p)$ and

$$
\|e^{-A_p t}u\|_{X_p}\le  e^{{ -\mu t}}\|u\|_{X_p}\quad {\rm for}\quad t\ge 0 \,\,\,{\rm and }\, \,\, u \in X_p.
 $$

 Let $X_p^{\alpha}=D(A_p^{\alpha})$ equipped with the graph norm  $\|u\|_{\alpha,p}:=\|u\|_{X_p^\alpha}=\|A_p^{\alpha}u\|_{L^p}$.

\begin{lemma}
\label{prelim-lm-5}
\begin{itemize}
\item[(i)] {Let $p \in (1,\infty).$} For each $\beta\ge 0$, there is $C_{p,\beta}>0$ such that for some $\gamma>0,$
\begin{equation*}
    \|A_p^\beta e^{-A_pt}\|_{L^p(\Omega)} \leq C_{p,\beta} t^{-\beta} e^{-\gamma t} \;\; \text{for}\;\; t>0.
\end{equation*}

\item[(ii)]
If $m \in \{0,1\}$ and $q\in [p,\infty]$ are such that $m-\frac{N}{q}<2\beta-\frac{N}{p}$,
then
$$
 X_p^\beta\hookrightarrow W^{m,q}(\Omega).
$$

\item[(iii)] If $2\beta -\frac{N}{p}>\theta\ge 0$, then
$$
X_p^\beta\hookrightarrow C^\theta(\Omega).
$$
\end{itemize}
\end{lemma}

\begin{proof}
(i) It follows from \cite[Theorem 1.4.3]{Hen}.

(ii) It follows from \cite[Theorem 1.6.1]{Hen}.

(iii) It also follows from \cite[Theorem 1.6.1]{Hen}.
\end{proof}

\begin{lemma}
\label{prelim-lm-4}
Let $\beta \geq 0,$ $p \in (1,\infty)$.  Then for any $\epsilon >0$ there exists $C_{p,\beta, \epsilon}>0$ such that for any $w \in  C^{\infty}_0(\Omega)$ we have
\begin{equation}
\label{001}
\|A_p^{\beta}e^{-tA_p}\nabla\cdot w\|_{L^p(\Omega)}  \leq C_{p,\beta, \epsilon} t^{-\beta-\frac{1}{2}-\epsilon} e^{-\gamma t} \|w\|_{L^p(\Omega)} \quad \text{for all}\,\,  t>0  \, \text{and  some } \, \gamma>0.
\end{equation}
Accordingly, for all $t>0$ the operator $A_p^\beta e^{-t A_p}\nabla\cdot $ admits a unique extension to all of $L^p(\Omega)$ which is again denoted by $A_p^\beta e^{-t A_p}\nabla\cdot$  and  satisfies $\eqref{001}$ for all $\mathbb{R}^n$-valued $w \in L^p(\Omega).$
\end{lemma}

\begin{proof}
It follows from \cite[Lemma 2.1]{HoWi}.
\end{proof}

Finally,  we present some basic properties of solutions of \eqref{main-eq}.
{ In the rest of this paper, if no confusion occurs, we put $A=A_p$ for some $1<p<\infty$.}

\begin{lemma}
\label{prelim-lm-01} {Assume $(1.5)'$.}
For any given $u_0\in C^0(\bar\Omega)$ satisfying ${(1.2)'}$ and $s\in\mathbb{R}$, there is a unique  classical solution $(u(t,x;s,u_0), v(t,x;s,v_0))$ of \eqref{main-eq}  on $(s,\infty)$ with initial condition $u(s,x;s,u_0)=u_0(x)$ (i.e.
$T_{\max}(s,u_0)=\infty$).   Moreover, $(u(t,x;s,u_0),v(t,x;s,u_0))$ satisfies
\begin{align}
\label{integral-eq1}
u(t,\cdot;s,u_0) &= e^{-A(t-s) } u_0 -\chi \int_s^t e^{ -A(t-\tau) }  \nabla \cdot \left(\frac{u(\tau,\cdot;s,u_0)}{v(\tau,\cdot;s,u_0)}  \nabla v(\tau,\cdot;s,u_0) \right) d\tau\nonumber\\
    &\quad + \int_s^t e^{  -A(t-\tau)}  u(\tau,\cdot;s,u_0)\big[{ \mu}+a(\tau,\cdot)-b(\tau,\cdot)u(\tau,\cdot;s,u_0)\big]  d\tau
\end{align}
for any $t>s$, and if
$\inf_{x\in\Omega}u_0(x)>0$, then for any $p>0$,
\begin{equation}
\label{proof-eqq-1}
\lim_{t\to s+}\int_\Omega u^{-p}(t,x;s,u_0)dx=\int_\Omega u^{-p}_0(x)dx.
\end{equation}
\end{lemma}

\begin{proof}  First, by the arguments of  \cite[Lemma 2.2]{FuWiYo1} and \cite[Theorem 1.2]{HKWS},
 for any given $u_0\in C^0(\bar\Omega)$ satisfying \eqref{initial-cond-eq} and $s\in\mathbb{R}$, there is a unique  classical solution $(u(t,x;s,u_0), v(t,x;s,v_0))$ of \eqref{main-eq}  on $(s,\infty)$
with initial condition $u(s,x;s,u_0)=u_0(x)$ and  $(u(t,\cdot;s,u_0), v(t,\cdot;s,v_0))$ satisfies
 \eqref{integral-eq1}.

Next, if $\inf_{x\in\bar\Omega}u_0(x)>0$, by \eqref{local-2-eq0},
$$
\lim_{t\to s+} u^{-p}(t,x;s,u_0)=u^{-p}_0(x)\quad \text{uniformly in}\,\, x\in\bar\Omega.
$$
This implies \eqref{proof-eqq-1} holds. The lemma is thus proved.
\end{proof}

To indicate the dependence of the classical solution $(u(t,x;s,u_0),v(t,x;s,u_0))$ of \eqref{main-eq} on $a(t,x)$ and $b(t,x)$, we may put
$$
(u(t,x;s,u_0,a,b),v(t,x;s,u_0,a,b))=(u(t,x;s,u_0),v(t,x;s,u_0)).
$$

\begin{lemma}
\label{prelim-lm-011}  Fix $s\in\RR$. Let  $u_n,u_0\in C^0(\bar\Omega)$ with $u_n(x)\ge 0$ and $a_n(t,x),b_n(t,x)$ satisfy {\bf (H)}. If $u_0$ satisfies  ${ (1.2)'}$,
  $\lim_{n\to\infty}\|u_n-u_0\|_\infty=0$, and
$$\lim_{n\to\infty} \sup_{x\in\Omega} |a_n(t,x)-a_0(t,x)|=0,\quad
\lim_{n\to\infty}  \sup_{x\in\Omega} |b_n(t,x)-b_0(t,x)|=0
$$
locally uniformly in $t\in\RR$
for some $a_0(t,x)$ and $b_0(t,x)$,
then  $a_0(t,x),b_0(t,x)$ satisfy {\bf (H)},  {there is $K>0$ such that $u_n$ satisfies \eqref{initial-cond-eq} for $n\ge K$,}    and
\begin{equation}
\label{proof-eqq-2}
\lim_{n\to\infty}\|u(t,\cdot;s,u_n,a_n,b_n)-u(t,\cdot;s,u_0,a_0,b_0)\|_\infty= 0\,\, {\rm uniformly\,\,  in}\,\,  t\in [s,s+T]
\end{equation}
  for any $T>0$.
\end{lemma}

\begin{proof}   This lemma is about the continuity of solutions of \eqref{main-eq} with respect to initials and the coefficients in the equations. {Note  that  solutions of \eqref{main-eq} satisfy the integral equation \eqref{integral-eq1}.
It  then  suffices to prove the continuity of solutions of \eqref{integral-eq1}  with respect to initials and the coefficients in the equations.}

First of all,  it is clear that $a_0(t,x)$ and $b_0(t,x)$ satisfy {\bf (H)}. Assume that
 $u_0 \in C^0(\bar \Omega)$ satisfies  \eqref{initial-cond-eq}.   For given $r>0$,  define
\begin{equation*}
    \mathcal{B}(u_0,r)=\{u\in C^0(\bar\Omega)\,|\, \|u-u_0\|_\infty\le r\}.
\end{equation*}
Fix $0<r\ll 1$  such that
\begin{equation}
\label{cont-wrt-initial-eq1}
\int_\Omega \tilde u_0(x)dx\ge \frac{2}{3}\int_\Omega u_0(x)dx\quad\text {for any}\,\, \tilde u_0\in \mathcal{B}(u_0,r).
\end{equation}
 Since $\lim_{n\to\infty}\|u_n-u_0\|_\infty=0$,  there is $K>0$ such that $u_n\in \mathcal{B}(u_0,r)$ and $\int_\Omega u_n(x)dx\ge \frac{2}{3}\int_\Omega u_0(x)dx>0$ for $n\ge K$.
This implies that $u_n$ satisfies {\bf (H)} for $n\ge K$.
Without loss of generality, we may assume that $u_n\in \mathcal{B}(u_0,r)$ for all $n\ge 1$.
Then $\int_\Omega u_n(x)dx>0$ for all $n\ge 1$. By Lemma \ref{prelim-lm-01},
$u(t,x;s,u_n,a_n,b_n)$ and $v(t,x;s,u_0,a_0,b_0)$ are defined for $t\ge s$.

Next, we note that it  suffices to prove that \eqref{proof-eqq-2} holds  for $0<T\ll 1$. To prove this,
 set  $$ \varepsilon:= \frac{\delta_0}{2} \int_{\Omega} u_0(x)dx  >0,$$
where $\delta_0$ is as in Lemma \ref{prelim-lm-00}.
For given $T>0$ and $R > r+\| u_0\|_{C^0(\bar \Omega)},$ we define the Banach space $$\mathcal{X}_T=C^0([s,s+T], C^0(\bar \Omega))$$
equipped with the norm
$$\|u\|_{\mathcal{X}_T}=\max_{ s \leq t \leq  s+T}\|u(t,x)\|_{{C^0}(\bar \Omega)}.
$$
Set
\begin{equation*}
    \mathcal{S}(T)=\left\{u \in \mathcal{X}_T: \|u\|_{\mathcal{X}_T} \leq R, \; \text{and} \; {\nu A^{-1}u} \ge \varepsilon \; \text{for all} \;t \in [s,s+T]\right \},
\end{equation*}
and
\begin{align*}
\mathcal{Y}(T)=\{a(\cdot,\cdot)\in \mathcal{X}_T \,:\,
|a(t,x)-a(t^{'},x)|\le B_1  |t-t^{'}|^\gamma,\,\,
 \alpha  \le a(t,x)\le B_2 \,\, \forall\, t,t^{'}\in [s,s+T],\,\, x\in\bar \Omega\},
\end{align*}
where $B_1,B_2,\alpha,\gamma$ are as in {\bf (H)}.
It is clear that  $\mathcal{S}(T)$ and $\mathcal{Y}(T)$  are closed subsets of the Banach space $\mathcal{X}_T$.

For $\tilde u_0 \in \mathcal{B}(u_0,r)$ and $(a,b)\in \mathcal{Y}(T)\times \mathcal{Y}(T)$,  we define
\begin{align*}
 {(  \mathcal{M}(\tilde u_0,a,b)u)(t,\cdot)}  &= e^{-A(t-s) }\tilde u_0 -\chi \int_s^t e^{  -A(t-\tau) }  \nabla \cdot \left(\frac{u(\tau,\cdot)}{v(\tau,\cdot)}  \nabla v(\tau,\cdot) \right) d\tau\nonumber\\
    &\quad + \int_s^t e^{  -A(t-\tau)}  u(\tau,\cdot)\big[{ \mu}+a(\tau,\cdot)-b(\tau,\cdot)u(\tau,\cdot)\big]  d\tau,
\end{align*}
where $u\in\mathcal{S}(T)$ and $v(\tau,\cdot)={ \nu} A^{-1}u(\tau,\cdot)$ for $\tau \in [s,s+T]$.  It is not difficult to prove that $\mathcal{M}(\tilde u_0,a,b)u\in \mathcal{X}_T$ for any $u\in\mathcal{S}(T)$.

We claim that   $\mathcal{M}(\tilde u_0,a,b)$ maps $\mathcal{S}(T)$ into itself for $0<T\ll 1$. To see this, let $p$,  $ \beta$, and $\epsilon>0$  be such that $N<p$, $\frac{N}{2p} < \beta <\frac{1}{2} $, and  $\epsilon \in (0,\frac{1}{2}-\beta).$
{Let $C_p$, $C_{p,\beta}$, and $C_{p,\beta,\epsilon}$ be as in
 by Lemma \ref{prelim-lm-000},  Lemma \ref{prelim-lm-5}, and  Lemma \ref{prelim-lm-4}, respectively}. Then for any $\tilde u_0 \in \mathcal{B}(u_0,r)$ and  $u\in \mathcal{S}(T)$, we have
 
\begin{align*}
&\|(\mathcal{M}(\tilde u_0,a,b)u)(t,\cdot)\|_{ C^0(\bar \Omega)} \\
 & \leq   \|e^{-A(t-s)}\tilde u_0\|_{C^0(\bar \Omega)}+{ \chi C_{p,\beta}} \int_{s}^{t}\| A^{\beta}e^{-A(t-\tau) }\nabla\cdot \Big( \frac{u(\tau,\cdot)}{v(\tau,\cdot)}  \nabla v(\tau,\cdot)\Big)\|_{L^p(\Omega)}d\tau\\
& \quad +{C_{p,\beta}} \int_{s}^{t}\| A^{\beta}e^{-A(t-\tau) } u(\tau,\cdot)\big[{ \mu}+a(\tau,\cdot)-b(\tau,\cdot)u(\tau,\cdot)\big] \|_{L^p(\Omega)} d\tau \\
&\leq  \|\tilde u_0\|_{C^0(\bar \Omega)} +\frac{\chi}{\varepsilon}{C_{p,\beta} C_{p,\beta,\epsilon} C_p|\Omega|^{1/p}} R^2\int_{s}^{t} (t-\tau)^{-\beta-\frac{1}{2}-\epsilon} d\tau \\
&\quad +{ C^2_{p,\beta}|\Omega|^{1/p}}(R{ \mu}+RB_2+R^2B_2) \int_{s}^{t} (t-\tau)^{-\beta} d\tau
\\
& \leq \|u_0\|_{C^0(\bar \Omega)}+r+ \frac{\chi}{\varepsilon}{ C_{p,\beta},C_{p,\beta,\epsilon} C_p|\Omega|^{1/p}} R^2T^{\frac{1}{2}-\beta-\epsilon}+ {C^2_{p,\beta}|\Omega|^{1/p}}R({ \mu}+B_2+R B_2)T^{1-\beta}
\end{align*}
for all $t \in [s,s+T].$ We then have $\|\mathcal{M}(u_0,a,b)u\|_{\mathcal{X}_T} \leq R$ if $T \in (0,1)$ is suitably small such that
\begin{equation}
\label{T-eq1}
    T \le \left( \frac{R-r-\| u_0\|_{C^0(\bar \Omega)}}{\frac{\chi}{\varepsilon}{C_{p,\beta}C_{p,\beta,\epsilon}C_p|\Omega|^{1/p}}R^2+{ C^2_{p,\beta}|\Omega|^{1/p}}R({\mu}+B_2+RB_2)}  \right)^{\frac{1}{\frac{1}{2}-\beta-\epsilon}}.
\end{equation}
Note that
{\begin{equation*}
    \int_\Omega { (e^{-A(t-s)}} \tilde u_0)(t,x)dx= { e^{-(t-s)} \int_\Omega e^{(t-s) \Delta}\tilde u_0(x)dx \ge e^{-T}}\int_\Omega \tilde u_0(x)dx,
\end{equation*}
\begin{align*}
    \int_\Omega \left[\int_{s}^{t} { e^{-A(t-\tau)}} \nabla\cdot \Big( \frac{u}{v}  \nabla v\Big)d\tau \right]dx &= {e^{-(t-s) }}\int_{s}^{t} \left[\int_\Omega e^{(t-\tau)\Delta} \nabla\cdot \Big( \frac{u}{v}  \nabla v\Big)dx\right]d\tau\\
    &= { e^{-(t-s)} } \int_{s}^{t} \left[\int_\Omega  \nabla\cdot \Big( \frac{u}{v}  \nabla v\Big)dx \right]d\tau\\
    &= { e^{-(t-s) }} \int_{s}^{t} 0 d\tau =0,
\end{align*}
and
\begin{align*}
    \int_\Omega \left[\int_{s}^{t} { e^{-A(t-\tau)} }u\big({ \mu}+a(\tau,\cdot)-b(\tau,\cdot)  u\big)d\tau\right]dx &= {e^{-(t-s)}} \int_{s}^{t} \left[\int_\Omega e^{(t-\tau) \Delta}u\big({ \mu}+a(\tau,\cdot)-b(\tau,\cdot)  u\big)dx \right]d\tau\\
    &= {e^{-(t-s)} } \int_{s}^{t} \left[\int_\Omega ({ \mu}+a(\tau,\cdot)) u-b(\tau,\cdot)  u^2dx \right]d\tau\\
    &\ge -B_2  |\Omega|   R^2 T { e^{-T}}.
\end{align*}}
It then follows that
\begin{align*}
    \int_\Omega (\mathcal{M}(\tilde u_0,a,b) u)(t,x)dx &= \int_\Omega { e^{-A(t-s)} } \tilde u_0(x)dx + \int_\Omega \left[\int_{s}^{t} {e^{-A(t-\tau)}} \nabla\cdot \Big( \frac{u(\tau,x)}{v(\tau,x)}  \nabla v(\tau,x)\Big)d\tau \right]dx\\
    &\quad + \int_\Omega \left[\int_{s}^{t} { e^{-A(t-\tau) } } u(\tau,x)\big({ \mu} +a(\tau,x)-b(\tau,x)  u(\tau,x)\big)d\tau\right]dx\\
    &\ge {e^{-T}} \int_\Omega  \tilde u_0(x)dx   -   {e^{-T}} B_2 |\Omega|R^2 T   \quad\forall t\in [s,s+T].
\end{align*}
This together with \eqref{cont-wrt-initial-eq1} implies that, if
\begin{equation}
\label{T-eq2}
0<{\frac{6T}{4-3e^T}}<\frac{1}{B_2|\Omega|R^2}\int_\Omega u_0(x)dx,
\end{equation}
then
$$
\int_\Omega( \mathcal{M}(\tilde u_0,a,b) u)(t,x)dx\ge \frac{1}{2}\int_\Omega u_0(x)dx\quad \forall \,\tilde u_0\in \mathcal{B}(u_0,r),\,\, (a,b)\in\mathcal{Y}_0(T),\,\,  u\in \mathcal{S}(T).
$$
Hence
\begin{equation*}
    ({\nu}A^{-1} \mathcal{M}(\tilde u_0,a,b)u)(t,x) \ge \delta_0  \int_\Omega \mathcal{M}(\tilde u_0,a,b)u(x,t)dx \ge  \frac{\delta_0}{2} \int_{\Omega}  u_0(x)dx = \varepsilon
\end{equation*}
for all $t\in [s,s+T].$   Therefore, the claim holds with any $T>0$ satisfies \eqref{T-eq1} and \eqref{T-eq2}.

\smallskip

We then prove that the mapping $\mathcal{M}(\tilde u_0,a,b)$ is a uniform contraction on $\mathcal{S}(T)$ for $0<T\ll 1$.
{Let $p$, $\beta$, $\epsilon$, and $C_p, C_{p,\beta}, C_{p,\beta,\epsilon}$ be as in the above.}  By similar arguments as in the above, for given $u,w \in \mathcal{S}(T)$ with { $v:=\nu A^{-1}u,$ $\bar v: =\nu A^{-1}w,$} we have
\begin{align}
\label{eq-M-cont-1}
    &\|(\mathcal{M}(\tilde u_0,a,b)u)(t,\cdot)-(\mathcal{M}(\tilde u_0,a,b) w)(t,\cdot)\|_{ C^0(\bar{\Omega})}\nonumber\\
    & \leq \chi { C_{p,\beta}}  \int_{s}^{t}\| { A^{\beta}e^{-A(t-\tau)} }\nabla\cdot \Big(\frac{u(\tau,\cdot)}{v(\tau,\cdot)}\nabla  v(\tau,\cdot)-\frac{w(\tau,\cdot)}{\bar v(\tau,\cdot)}\nabla \bar v(\tau,\cdot)\Big)  \|_{L^p(\Omega)}d\tau \nonumber\\
    & \quad +{ C_{p,\beta}}\int_{s}^{t}\| { A^{\beta}e^{-A(t-\tau) }}(u(\tau,\cdot)-w(\tau,\cdot))\big[{ \mu}+a(\tau,\cdot)-b(\tau,\cdot) (u(\tau,\cdot)+w(\tau,\cdot))\big]\|_{ L^p(\Omega)} d\tau \nonumber \\
    & \leq \frac{ \chi { C_{p,\beta} C_{p,\beta,\epsilon}C_p|\Omega|^{1/p} }}{\varepsilon} \int_{s}^{t} (t-\tau)^{-\beta-\frac{1}{2}-\epsilon}  \|u(\tau)-w(\tau)\|_{C^0(\Omega)}\nonumber\\
&\qquad\qquad\qquad \qquad\qquad\qquad  \cdot \Big(\|u(\tau)\|_{C^0(\Omega)}+\|w(\tau)\|_{C^0(\Omega)}+\frac{{ C_{2p}}}{\varepsilon}\|w(\tau)\|^2_{C^0(\Omega)}\Big) d\tau \nonumber\\
    & \quad +{C^2_{p,\beta}|\Omega|^{1/p}}\int_{s}^{t} (t-\tau)^{-\beta}\| u(\tau)-w(\tau)\|_{ C^0(\Omega)}\Big({ \mu}+B_2+B_2\big(\|u(\tau)\|_{ C^0(\Omega)}+\|w(\tau)\|_{ C^0(\Omega)}\big)\Big) d\tau \nonumber \\
    & \leq T^{\frac{1}{2}-\beta-\epsilon}
\Big( \frac{\chi {C_{p,\beta} C_{p,\beta,\epsilon}C_p|\Omega|^{1/2}}}{\varepsilon^2} (2 \varepsilon R+ { C_{2p}}R^2 ) +{C^2_{p,\beta}|\Omega|^{1/p}}({ \mu}+B_2+2RB_2)\Big)\|u-w\|_{\mathcal{X}}
\end{align}
for all $t \in [s,s+T]$. It then follows that $\mathcal{M}(\tilde u_0,a,b)$ is a uniform contraction on $\mathcal{S}(T)$ for any $T>0$ satisfying \eqref{T-eq1}, \eqref{T-eq2}, and
\begin{equation}
\label{T-eq3}
    T < \left[\frac{\chi { C_{p,\beta} C_{p,\beta,\epsilon} C_p}|\Omega|^{1/p}}{\varepsilon^2} (2\varepsilon R+{ C_{2p}}R^2)+{ C_{p,\beta}^2|\Omega|^{1/p}}({ \mu}+B_2+2RB_2) \right]^{\frac{-1}{\frac{1}{2}-\beta-\epsilon}}.
\end{equation}

We now prove $\mathcal{M}({\tilde  u_0},a,b)$  is uniformly continuous in $\tilde u_0\in\mathcal{B}(u_0,r)$ and
$(a,b)\in \mathcal{Y}(T)\times \mathcal{Y}(T)$.  Note that for any $\tilde u_1,\tilde u_2\in \mathcal{B}(u_0,r)$, $(a_1,b_1), (a_2,b_2)\in\mathcal{Y}(T)\times \mathcal{Y}(T)$, and $u\in\mathcal{S}(T)$,
\begin{align*}
\mathcal{M}(\tilde u_2,a_2,b_2)u-\mathcal{M}(\tilde u_1,a_1,b_1)u=&[\mathcal{M}(\tilde u_2,a_2,b_2)u-\mathcal{M}(\tilde u_1,a_2,b_2)u ]\\
&+[\mathcal{M}(\tilde u_1,a_2,b_2)u-\mathcal{M}(\tilde u_1,a_1,b_2)u]\\
&+[\mathcal{M}(\tilde u_1,a_1,b_2)u-\mathcal{M}(\tilde u_1,a_1,b_1)u].
\end{align*}
It then suffices to prove that $\mathcal{M}(\tilde u_0,a,b)$ is uniformly  continuous in $\tilde u_0\in\mathcal{B}(u_0,r)$
(resp. uniformly continuous  in $a\in\mathcal{Y}(T)$, uniformly continuous in $b\in\mathcal{Y}(T)$).
Observe that, for any  $\tilde u_n, \tilde u_0 \in \mathcal{B}(u_0,r)$, $(a,b)\in\mathcal{Y}_0(T)\times\mathcal{Y}_0(T)$, and
$u\in \mathcal{S}(T)$,   by comparison principles for parabolic equations, we have
\begin{equation*}
    \|(\mathcal{M}(\tilde u_n,a,b)u)(t,\cdot)-(\mathcal{M}({\tilde  u_0},a,b)u)(t,\cdot)\|_{ C^0(\bar{\Omega})} = \|e^{-(t-s)A}(\tilde u_n-\tilde u_0)\|_{C^0(\bar\Omega)}\le \|\tilde u_n-\tilde u_0\|_{C^0(\bar\Omega)}.
\end{equation*}
Thus,  the mapping   $\mathcal{M}({\tilde  u_0},a,b)$  is uniformly  continuous in $\tilde u_0\in\mathcal{B}(u_0,r)$.
For any $\tilde u_0\in\mathcal{B}(u_0,r)$, $a_n,a_0,b\in\mathcal{Y}(T)$, and $u\in\mathcal{S}(T)$,
\begin{align*}
\|(\mathcal{M}(\tilde u_0,a_n,b)u)(t,\cdot)-(\mathcal{M}(\tilde u_0,a_0,b)u)(t,\cdot)\|_{C^0(\bar\Omega)}&= \| \int_s^t e^{  -A(t-\tau)}  u(\tau,\cdot)\big[a_n(\tau,\cdot)-a_0(\tau,\cdot)\big]  d\tau\|_{C^0(\bar\Omega)}\\
&\le R \|a_n-a_0\|_{\mathcal{X}_T} T.
\end{align*}
This implies that  the mapping   $\mathcal{M}({\tilde  u_0},a,b)$  is uniformly continuous in $a\in\mathcal{Y}(T)$.
For any $\tilde u_0\in\mathcal{B}(u_0,r)$, $a,b_n,b_0\in\mathcal{Y}(T)$, and $u\in\mathcal{S}(T)$,
\begin{align*}
\|(\mathcal{M}(\tilde u_0,a,b_n)u)(t,\cdot)-(\mathcal{M}(\tilde u_0,a,b_0)u)(t,\cdot)\|_{C^0(\bar\Omega)}&= \| \int_s^t e^{  -A(t-\tau)}  u^2(\tau,\cdot)\big[b_n(\tau,\cdot)-b_0(\tau,\cdot)\big]  d\tau\|_{C^0(\bar\Omega)}\\
&\le R^2 \|b_n-b_0\|_{\mathcal{X}_T} T.
\end{align*}
Therefore,   the mapping   $\mathcal{M}({\tilde  u_0},a,b)$  is also uniformly continuous in $b\in\mathcal{Y}(T)$.

Finally, Let $T>0$ satisfy \eqref{T-eq1}, \eqref{T-eq2}, and \eqref{T-eq3}.
By Lemmas \ref{cont-family} and \ref{prelim-lm-01}, we conclude that $\mathcal{M}({\tilde  u_0},a,b)$ has a unique fixed point $u \in \mathcal{S}(T)$ fulfilling $(\mathcal{M}({\tilde  u_0},a,b)u)(t,\cdot)=u(t,\cdot;s,{\tilde  u_0},a,b)$ for $t \in [s,s+T],$ and we also have that $u(t,\cdot;s,{\tilde u_0},a,b)\in C^0(\bar\Omega)$ is continuous in ${\tilde  u_0}\in\mathcal{B}(u_0,r)$ and $(a,b)\in\mathcal{Y}(T)$ uniformly in  $t \in [s,s+T].$ The lemma is thus proved.
\end{proof}

\section{Boundedness of positive solutions}

In this section, we investigate the boundedness of positive solutions of \eqref{main-eq}, and  prove Theorem \ref{main-thm1}.  {Throughout this section, we assume that $(1.5)'$ holds.}

For given $s\in\mathbb{R}$ and $u_0$ satisfying \eqref{initial-cond-eq}, put
 $(u(t,x),v(t,x))=(u(t,x;s,u_0),v(t,x;s,u_0))$ for $t\ge s$ and $x\in\bar\Omega$.
Observe that, for any $1<q<\infty$,
\begin{align*}
\frac{1}{q}\int_\Omega u_t^q(t,x)dx& =-(q-1)\int_\Omega u^{q-2}(t,x)|\nabla u(t,x)|^2dx -\chi(q-1)\int_\Omega \frac{u^{q-1}(t,x)}{v(t,x)}\nabla u(t,x)\cdot\nabla v(t,x)dx\nonumber\\
&\,\,\,\, +\int_\Omega a(t,x)u^q(t,x)-\int_\Omega b(t,x)u^{q+1}(t,x) dx\quad \forall\, {t>s},
\end{align*}
and
 for any $p>0$,
\begin{align}
\label{proof-bdd-eq00}
    \frac{1}{p} \cdot \frac{d}{dt} \int_{\Omega} u^{-p}(t,x)= & -(p+1)\int_{\Omega} u^{-p-2}(t,x)|\nabla u(t,x)|^2 + (p+1) \chi \int_{\Omega} \frac{u^{-p-1}(t,x)}{v(t,x)}\nabla u (t,x)\cdot \nabla v(t,x) \nonumber\\
&-  \int_{\Omega} a(t,x) u^{-p} (t,x)+  \int_{\Omega} b(t,x) u^{-p+1}(t,x)\quad \forall\, {t>s}.
\end{align}
In the rest of this section, we may omit $(t,x)$ inside the integrals if no confusion occurs.

We first present some lemmas.

\begin{lemma}
\label{main-lem0}
For any $s\in\mathbb{R}$ and  $u_0$ satisfying ${ (1.2)'}$,
\begin{equation*}
    \int_{\Omega} u(t,x)dx \leq m^*(\tau,s,u_0)= {\rm max}\Big\{\int_{\Omega} u(\tau,x)dx,  \frac{a_{\sup}}{b_{\inf}}|\Omega| \Big\} \quad \forall\,\, t>\tau\ge s
\end{equation*}
and
$$
\limsup_{t-s\to\infty}\int u(t,x)dx\le \frac{a_{\sup}}{b_{\inf}}|\Omega|,
$$
where $|\Omega|$ is the Lebesgue measure of $\Omega$.
\end{lemma}

\begin{proof}
By integrating the first equation in \eqref{main-eq} with respect to $x$, we get that
\begin{align*}
\frac{d}{dt} \int_{\Omega}u&=\int_{\Omega}\Delta u- \chi  \int_{\Omega} \nabla \cdot \Big(\frac{u}{v}\nabla v \Big)+ \int_{\Omega} a(x,t)u -b(x,t) u^2\\
    & =  \int_{\Omega}a(x,t)u(x,t)dx - \int_{\Omega} b(x, t)u^2(x,t)dx\\
    &  \leq a_{\sup} \int_{\Omega}u - \frac{b_{\inf}}{|\Omega|} \Big(\int_{\Omega} u\Big)^2 .
\end{align*}
This together with comparison principle for scalar ODEs  implies that
\begin{equation*}
    \int_{\Omega} u(t,x)dx \leq  {\rm max}\Big\{\int_{\Omega} u(\tau,x)dx,  \frac{a_{\sup}}{b_{\inf}}|\Omega| \Big\} \quad \forall\,\, t>\tau\ge s
\end{equation*}
and
$$
\limsup_{t-s\to\infty}\int u(t,x)dx\le \frac{a_{\sup}}{b_{\inf}}|\Omega|.
$$
The lemma is thus proved.
\end{proof}

\begin{lemma}
\label{main-lem2}
Let $q \ge  3$ and $q-1 \leq k <  2q-2.$  There exist positive constants  $ M(q,k)>0$ and $M^{*}(q,k)>0$  such that
\begin{equation*}
   \int_{\Omega} \frac{|\nabla v|^{2q}}{v^{k+1}} \leq M(k,q) \int_{\Omega} \frac{u^q}{v^{k-q+1}}+M^*(k,q)\int_{\Omega} v^{2q-k-1}
\end{equation*}
for any $s\in\mathbb{R}$, $u_0\in C^0(\bar\Omega)$ satisfying ${(1.2)'}$, and
 ${t\in (s, \infty)}$.
\end{lemma}

\begin{proof}
It follows from \cite[{Proposition 1.3}]{HKWS}.
\end{proof}

\begin{lemma}
\label{main-lem3}
For any $p>0$, $s\in\mathbb{R}$ and $u_0\in C^0(\bar\Omega)$ satisfying ${(1.2)'}$,
\begin{equation*}
    \int_{\Omega} u \ge |\Omega|^{\frac{p+1}{p}} \left(\int_{\Omega}u^{-p}\right)^{-\frac{1}{p}} \quad \forall\, t>s.
\end{equation*}
\end{lemma}
\begin{proof}
For any given $p>0$, by H\"older's inequality, we have
\begin{align*}
|\Omega|=\int_\Omega u^{\frac{p}{p+1}} u^{-\frac{p}{p+1}}\le \Big(\int_\Omega u\Big)^{\frac{p}{p+1}}\Big(\int_\Omega u^{-p}\Big)^{\frac{1}{p+1}}\quad \forall\, t>s.
\end{align*}
This implies that \begin{equation*}
    \int_{\Omega} u \ge |\Omega|^{\frac{p+1}{p}} \left(\int_{\Omega}u^{-p}\right)^{-\frac{1}{p}} \quad \forall\, t>s.
\end{equation*}
The lemma is thus proved.
\end{proof}

\begin{lemma}
\label{main-lem4}
Let $p>0$. Then for every $\beta>0$, we have
\begin{align*}
     (p+1)\chi \int_{\Omega}  \frac{u^{-p-1}}{v} \nabla u \cdot \nabla v &\leq (p+1) \int_{\Omega} u^{-p-2}|\nabla u|^2 + \frac{(p+1)\beta\mu}{p} \int_{\Omega} u^{-p} \nonumber\\
     & \quad + \left[\frac{ (p+1)(\chi-\beta)^2 }{4}-\frac{(p+1)\beta}{p} \right] \int_{\Omega} u^{-p} \frac{|\nabla v|^2}{v^2}
\end{align*}
for all $t\in (s,\infty)$.
\end{lemma}

\begin{proof}
This is Lemma \cite[Lemma 3.3]{HKWS}.
\end{proof}

\begin{lemma}
\label{main-lem5}
Let $R>0$ be such that
\begin{equation*}
    R>
\begin{cases}
\frac{\mu \chi^2}{4}, &\text{if\,\, $0< \chi \leq 2$}\\
\mu(\chi-1), &\text{if \,$\chi>2$.}\\
\end{cases}
\end{equation*}
Then there is $\beta>0$, $\beta\not =\chi$ such that
\begin{equation*}
    \frac{(p+1)\beta \mu}{p}-R<0,
\end{equation*}
where $p$ is given by
\begin{equation*}
    p=\frac{4\beta}{( \chi-\beta)^2}.
\end{equation*}
\end{lemma}

\begin{proof}
This is \cite[Lemma 3.1]{HKWS}.
\end{proof}

We now prove Theorem \ref{main-thm1}.

\begin{proof}[Proof of Theorem \ref{main-thm1}]
(1) { It follows from the arguments of \cite[Theorem 1.1(3)]{HKWS}.
To be more precise,  let
$\beta=\chi+2-2\sqrt{\chi+1}$. Then
\begin{equation*}
    p:=\frac{4\beta}{(\chi-\beta)^2}=\frac{4(\chi+2-2\sqrt{\chi+1})}{(-2+2\sqrt{\chi+1})^2}=1 \,\,\Longrightarrow \,\, \frac{(\chi-\beta)^2}{4}-\beta=0,
\end{equation*}
and
\begin{equation*}
    \frac{(p+1)\beta\mu}{p}=2\beta\mu=2(\chi+2-2\sqrt{\chi+1})\mu=a_{\chi,\mu}<a_{\inf}.
\end{equation*}
By \eqref{proof-bdd-eq00},
\begin{align*}
     \cdot \frac{d}{dt} \int_{\Omega} u^{-1} \leq -2\int_{\Omega} u^{-3}|\nabla u|^2 + 2 \chi \int_{\Omega} \frac{u^{-2}}{v}\nabla u \cdot \nabla v
     -  a_{\inf} \int_{\Omega} u^{-1} +  b_{\sup} |\Omega|
\end{align*}
for all $t>s$.
 By Lemma \ref{main-lem4},  we have that
\begin{align*}
      \cdot \frac{d}{dt} \int_{\Omega} u^{-1} &\leq 2 \left[\frac{ (\chi-\beta)^2 }{4}-{\beta} \right] \int_{\Omega} u^{-1} \frac{|\nabla v|^2}{v^2} + \left[2\beta\mu-a_{\inf}\right] \int_{\Omega} u^{-1}+ b_{\sup} |\Omega|\\
& = -(a_{\inf}-a_{\chi,\mu}) \int_{\Omega} u^{-1} + b_{\sup}|\Omega|,\,\, \,\, \forall\, t>\tau>s.
\end{align*}
It then follows from comparison principle for scalar ODEs that
\begin{equation*}
     \int_\Omega u^{-1}(t,x;s,u_0)dx\le e^{-(a_{\inf}-a_{\chi,\mu}) (t-\tau)} \int_{\Omega} u^{-1}(\tau,x;s,u_0) dx +\frac{b_{\sup}|\Omega|}{a_{\inf}-a_{\chi,\mu}}
\end{equation*}
and
\begin{equation}
\label{new-new-eq1-1}
\int_\Omega u^{-1}(t,x;s,u_0)dx\le \max\Big\{\int_\Omega u^{-1}(\tau,x;s,u_0)dx,  \frac{b_{\sup}|\Omega|}{a_{\inf}-a_{\chi,\mu}}\Big\}
\end{equation}
for any $t>\tau>s$.  Moreover,
$$
\lim_{t-s\to\infty}\int_\Omega u^{-1}(t,x;s,u_0)dx=\frac{b_{\sup}|\Omega|}{a_{\inf}-a_{\chi,\mu}}.
$$
 Theorem \ref{main-thm1}(1)  thus follows. }

(2)  { It also follows from the  arguments of \cite[Theorem 1.1(3)]{HKWS}.
To be more precise, by Lemmas \ref{prelim-lm-00} and \ref{main-lem3}, we have
\begin{equation*}
  v(t,x)\ge \delta_0  \int_{\Omega} u \ge \delta_0  |\Omega|^{2} \left(\int_{\Omega}u^{-1}\right)^{-1} \quad \forall\, t>s.
\end{equation*}
By  \eqref{new-new-eq1-1} and  the assumption $(1.2)'$
\begin{equation*}
        \int_\Omega u\ge \frac{(a_{\inf}-a_{\chi,\mu})\min\{1,\chi\}}{b_{\sup}|\Omega|},\quad
v\ge \frac{\delta_0  (a_{\inf}-a_{\chi,\mu})\min\{1,\chi\}|\Omega|}{b_{\sup}}
\quad \text{for all}\,\,   t\in [s+\tau_0,\infty).
\end{equation*}
By the arguments of  \cite[Theorem 1.1 (3)]{HKWS}, for any $q>\max\{2,N\}$ and any $\varepsilon>0$, there is $C(\varepsilon,q)>0$ such that
\begin{align*}
    \frac{1}{q} \cdot \frac{d}{dt} \int_{\Omega} u^q
    &\leq \frac{b_{\sup}|\Omega| (q-1)\Big\{(C^*)^{\frac{1}{q+1}}+\varepsilon\Big\}\max\{\chi,\chi^2\}}{4\delta_0  (a_{\inf}-a_{\chi,\mu})} \int_\Omega u^{q+1} \\
&\quad
+ \frac{b_{\sup}|\Omega|(q-1)  C(\varepsilon,p)\max\{\chi,\chi^2\}}{4\delta_0 (a_{\inf}-a_{\chi,\mu})}
    \Big(\int_\Omega u\Big)^{{q+1}} \\
&\quad +a_{\sup}\int_{\Omega}u^q -  b_{\inf} \int_{\Omega} u^{q+1}
    \,\, \forall\,  t\in (s+\tau_0,\infty).
\end{align*}
By the assumption $(1.5)'$,  there  are $q>\max\{2,n\}$ and $0<\varepsilon=\varepsilon(q)\ll 1$ such that
\begin{equation*}
    \frac{b_{\sup}|\Omega| (q-1)\Big\{ (C^*(n,\nu))^{\frac{1}{q+1}}+ \varepsilon\Big\}\max\{\chi,\chi^2\}}{4\delta_0 (a_{\inf}-a_{\chi,\mu})}  <b_{\inf}.
\end{equation*}
Note that
\begin{equation*}
    \int_\Omega u^{q+1}\ge \frac{1}{|\Omega|}\Big(\int_\Omega u^q\Big)^{\frac{q+1}{p}}\quad \text{for all}\,\,  t\in (s,\infty).
\end{equation*}
Therefore, there is $b_1>0$ such that
\begin{align*}
    \frac{1}{q} \cdot \frac{d}{dt} \int_{\Omega} u^q
    &\leq  \frac{b_{\sup}|\Omega|(q-1) C(\varepsilon,q)\max\{\chi,\chi^2 \}}{4\delta_0 (a_{\inf}-a_{\chi,\mu})}
\Big(\int_\Omega u\Big)^{{ q+1}}\nonumber\\
&\quad  +a_{\sup}\int_{\Omega}u^q -  b_1\Big( \int_{\Omega} u^{q}\Big)^{\frac{q+1}{p}}
\quad \text{for all}\,\,  t\in (s+\tau_0,\infty).
\end{align*}
This  together with    Lemma \ref{main-lem0} implies that there are  $M_2(1,q)$ and $\tilde M_2(1,q,\tau,s,u_0)$
such that
\begin{equation}
\label{tilde-M-2}
\begin{cases}
\frac{d}{dt}\int_\Omega u^{q}\le -\int _\Omega u^q+M_2(1,q)\quad \forall\, t\ge s+\tau_0\cr
\frac{d}{dt}\int_\Omega u^{q}\le -\int _\Omega u^q+M_2(1,q)+
\tilde M_2(1,q,\tau,s,u_0)
\quad \forall\, t\ge \tau> s.
\end{cases}
\end{equation}
Theorem \ref{main-thm1}(2) then follows. }
\end{proof}

\section{Globally attracting invariant rectangle and pointwise persistence}

In this section, we investigate the  existence of globally attracting invariant rectangle for \eqref{main-eq} and pointwise persistence, and  prove Theorems \ref{main-thm2} and \ref{main-thm3}.

We first prove a lemma.

\begin{lemma}
\label{prelim-lm-2}
Assume that ${ (1.5)'}$
holds. Let $p>0$ be as in Theorem \ref{main-thm1}(1).  There are $M_1^*>0$ and
$\tilde M_1^*>0$ such that for every $u_0\in C^0(\bar\Omega)$ satisfying ${ (1.2)'}$  and $\int_\Omega u_0(x)dx\le \frac{a_{\sup}}{b_{\inf}}|\Omega|$, $s\in\mathbb{R}$, and $\tau>0$,
\begin{equation}
    \label{bdd-p-int}
    \int_\Omega u^{-p}(t,x;s,u_0)dx \leq \max \Big \{\int_{\Omega} u^{-p}(s+\tau,x;s,u_0)dx,M_1^* \Big\} \quad \text{for all} \,\,  s+\tau<t<\infty
\end{equation}
and
\begin{equation}
\label{bdd-p-int-0}
\limsup_{t-s\to\infty} \int_\Omega u^{-p}(t,x;s,u_0)dx\le M_1^*,
\end{equation}
\begin{equation}
\label{bdd-p-int-00}
\liminf_{t-s\to\infty}\int_\Omega u(t,x;s,u_0)dx\ge \tilde M_1^*,\quad \liminf_{t-s\to\infty} \inf_{x\in\Omega}v(t,x;s,u_0)\ge \delta_0 \tilde M_1^*.
\end{equation}
In addition, if $\inf_{x\in\Omega}u_0(x)>0$ and $\int_\Omega u_0^{-p}(x)dx\le M_1^*$, then
\begin{equation}
    \label{bdd-p-int-1}
    \int_\Omega u^{-p}(t,x;s,u_0)dx \leq  M_1^*\quad   \forall \,\,  s<t<\infty
\end{equation}
and
\begin{equation}
\label{bdd-p-int-1-1}
\int_\Omega u(t,x;s,u_0)\ge \tilde M_1^*, \quad  \inf_{x\in\Omega} v(t,x;s,u_0)\ge \delta_0 \tilde M_1^*\quad \forall t\in [s,\infty).
\end{equation}
\end{lemma}

\begin{proof}
First,
let $M_1^*=M_1(p)$.  Then, by Lemma \ref{main-lem0},
$$
\int_\Omega u(\tau,x;s,u_0)dx\le \frac{a_{\sup}}{b_{\inf}}|\Omega|\quad \forall \,\, \tau\ge s.
$$
Hence
\begin{equation}
\label{new-new-eq11}
m^*(\tau,s,u_0)=\frac{a_{\sup}}{b_{\inf}}
\quad {\rm and}\quad
\tilde M_1(p,\tau,s,u_0)=0\quad \forall\,\, \tau\ge s.
\end{equation}
This together with \eqref{revised-negative-p-eq1} and   \eqref{revised-negative-p-eq2} implies  \eqref{bdd-p-int} and \eqref{bdd-p-int-0} hold.

Next,
 by H\"older inequality, we have
\begin{align}
\label{bdd-p-int-000}
|\Omega|&=\int_\Omega u^{\frac{p}{p+1}}(t,x;s,u_0) u^{-\frac{p}{p+1}}(t,x;s,u_0)dx\nonumber \\
&\le \big(\int_\Omega  u(t,x;s,u_0)dx\big)^{\frac{p}{p+1}} \big(\int_\Omega u^{-p}(t,x;s,u_0)dx \big)^{\frac{1}{p+1}}
\quad\forall\, t> s.
\end{align}
This together with \eqref{bdd-p-int-0} and Lemma  \ref{prelim-lm-00} implies that  \eqref{bdd-p-int-00} holds with
$\tilde M_1^*= \frac{|\Omega|^{\frac{p+1}{p}}}{(M_1^*)^{\frac{1}{p}}}$.

Now,  assume $\inf_{x\in\Omega}u_0(x)>0$ and $\int_\Omega u_0^{-p}(x)dx\le M_1^*$.
 By  \eqref{proof-eqq-1},
$$
\lim_{\tau \to 0}\int_\Omega u^{-p}(s+\tau,x;s,u_0)dx =\int_\Omega u_0^{-p}(x)dx\le M_1^*.
$$
This together with \eqref{bdd-p-int} implies \eqref{bdd-p-int-1}.
\eqref{bdd-p-int-1-1} then follows from  \eqref{bdd-p-int-1}, \eqref{bdd-p-int-000}, and
 Lemma  \ref{prelim-lm-00}. The lemma is thus proved.
\end{proof}

Next, we prove  Theorem \ref{main-thm2}.

\begin{proof}[Proof of Theorem \ref{main-thm2}]
Let $p$ and $M_1(p)$ be as in Theorem \ref{main-thm1}(1), and $q$ and $M_2(p,q)$ be as in Theorem \ref{main-thm1}(2).
Let $M_0^*=\frac{a_{\sup}}{b_{\inf}}|\Omega|$, $M_1^*=M_1(p)$, and $M_2^*=M_2(p,q)$.
 In the following, we prove that (1) and (2)  in Theorem \ref{main-thm2}  hold with these $p,q,M_1^*, M_2^*$ and some $M_3^*(\tau,\theta)$, $M_4^*(\theta)$.

\medskip

(1) We first prove that the set $\mathcal{E}$ is invariant.

First of all, for any $u_0\in \mathcal{E}$,
\begin{equation*}
|\Omega|^2=\left(\int_{\Omega} u_0^{-p/2}(x) u_0^{p/2}(x)dx\right)^2\le \int_\Omega u_0^{-p}(x) dx\cdot\int_\Omega u_0^p(x)dx\le M_1^*\int_\Omega u_0^p(x)dx.
\end{equation*}
This implies that $\int_\Omega u_0^p(x)dx>0$ and then $\int_\Omega u_0(x)dx>0$.
{Recall that $p=1$ and
$$M_1^*=M_1(p)=\frac{b_{\sup}|\Omega|}{a_{\inf}-a_{\chi,\mu}}\le  \frac{b_{\sup}|\Omega|\max\{1,\frac{1}{\chi}\}}{a_{\inf}-a_{\chi,\mu}},
$$
Hence $u_0$ satisfies $(1.2)'$ with $\tau_0=0$.}

For any $s\in\mathbb{R}$ and $u_0\in\mathcal{E}$, by Lemma \ref{main-lem0},
\begin{equation}
\label{new-new-eq12}
\int_\Omega u(t,x;s,u_0)dx\le \max\left\{\int_\Omega u_0(x)dx,M_0^*\right\}=M_0^*\quad \forall\, t\ge s.
\end{equation}
{ By Lemma \ref{prelim-lm-2}, there holds
\begin{equation}
\label{new-eqqq6}
\int_\Omega u{^{-p}}(t,x;s,u_0)dx\le M_1^*\quad \forall\, t>s.
\end{equation}
 By \eqref{revised-negative-p-eq1},}
\eqref{tilde-M-2} { with $\tau_0=0$}, \eqref{new-new-eq11}, and \eqref{new-eqqq6},
\begin{align*}
{  \int_\Omega u^q(t,x;s,u_0)dx}
&{ \le  \max\left\{\int_\Omega u^q(\tau,x;s,u_0)dx, M_2^*\right\}\quad \forall\, t>\tau>s.}
\end{align*}
Letting $\tau\to s^+$, we have
\begin{equation}
\label{new-eqqq7}
\int_\Omega u^q(t,x;s,u_0)dx\le M_2^*\quad \forall\, t\ge s.
\end{equation}
It then  follows from \eqref{new-new-eq12},  \eqref{new-eqqq6} and \eqref{new-eqqq7} that the set $\mathcal{E}$ is invariant.

Next, we  prove \eqref{new-set-eq1} holds. To this end,  for any given $0<\theta<1-\frac{2N}{q}$, choose $\beta\in (0,\frac{1}{2})$ such that
$$
2\beta-\frac{2N}{q}>\theta.
$$
By Lemma \ref{prelim-lm-5}, there is $K_1$ such that for any $u_0\in\mathcal{E}$ and $s\in\RR$,
\begin{equation}
\label{new-thm2-eq1}
\|u(t,\cdot;s,u_0)\|_{C^\theta(\bar\Omega)}\le K_1 \|u(t,\cdot;s,u_0)\|_{X^\beta_{q/2}}\quad \forall\, t>s.
\end{equation}
Note that
 for any $u_0 \in \mathcal{E}$ and $s\in\RR$,
\begin{align}
\label{new-thm2-eq2}
    \|u(t,\cdot;s,u_0)\|_{X^{\beta}_{q/2}} & \leq   \|A^{\beta}e^{-A(t-s)} u_0\|_{L^{q/2}(\Omega)}\nonumber\\
   &\quad +\chi \int_{s}^{t}\left\| A^{\beta}e^{-A(t-\lambda)}\nabla\cdot \left(\frac{u(\lambda,\cdot;s,u_0)}{v(\lambda,\cdot;s,u_0)}  \nabla v(\lambda,\cdot;s,u_0) \right)\right\|_{L^{q/2}(\Omega)}d\lambda\nonumber \\
    & \quad + \int_{s}^{t}\| A^{\beta}e^{-A(t-\lambda)} \big((1+a(\lambda,\cdot))u(\lambda,\cdot;s,u_0)- b(\lambda) u^2(\lambda, \cdot;s,u_0)\big)\|_{L^{q/2}(\Omega)}d\lambda,
\end{align}
where   $A=A_{q/2}$. By Lemma \ref{prelim-lm-5} again, there is $K_2$ such that
\begin{equation}
\label{new-thm2-eq3}
 \|A^{\beta}e^{-A(t-s)} u_0\|_{L^{q/2}(\Omega)}\le K_2 (t-s)^{-\beta(t-s)} e^{-\gamma(t-s)}\|u_0\|_{L^{q/2}}\quad \forall\, t>s.
\end{equation}
By Lemma \ref{prelim-lm-4},  there is $K_3$ such that
\begin{align}
\label{new-thm2-eq4}
&\chi \int_{s}^{t}\left\| A^{\beta}e^{-A(t-\lambda)}\nabla\cdot \left(\frac{u(\lambda,\cdot;s,u_0)}{v(\lambda,\cdot;s,u_0)}  \nabla v(\lambda,\cdot;s,u_0) \right)\right\|_{L^{q/2}(\Omega)}d\lambda\nonumber \\
    & \leq  K_3 \chi  \int_s^{t}  (t-\lambda)^{-\beta-\frac{1}{2}-\epsilon}e^{-\gamma(t-\lambda)}  \left\| \frac{u(\lambda,\cdot;s,u_0)}{v(\lambda,\cdot;s,u_0)} \cdot \nabla v(\lambda,\cdot;s,u_0) \right\|_{L^{q/2}(\Omega)} d\lambda\nonumber \\
&\le \frac{K_3\chi}{\tilde M_1^*}  \int_s^{t}  (t-\lambda)^{-\beta-\frac{1}{2}-\epsilon}e^{-\gamma(t-\lambda)}  \left\| {u(\lambda,\cdot;s,u_0)} \cdot \nabla v(\lambda,\cdot;s,u_0) \right\|_{L^{q/2}(\Omega)} d\lambda\quad  \text{(by Lemma \ref{prelim-lm-2})}\nonumber \\
&\le \frac{K_3\chi}{\tilde M_1^*}  \int_s^{t}  (t-\lambda)^{-\beta-\frac{1}{2}-\epsilon}e^{-\gamma(t-\lambda)}  \left\| {u(\lambda,\cdot;s,u_0)}\right\|_{L^q(\Omega)}\left\| \nabla v(\lambda,\cdot;s,u_0) \right\|_{L^{q}(\Omega)} d\lambda\quad  \text{(by H\"older inequality)}\nonumber \\
&\le   \frac{CK_3\chi (M_2^*)^{2/q} }{\tilde M_1^*}  \int_s^{t}  (t-\lambda)^{-\beta-\frac{1}{2}-\epsilon}e^{-\gamma(t-\lambda)}  d\lambda\quad  \text{(by Lemma \ref{prelim-lm-000} and Theorem \ref{main-thm1})}.
\end{align}
By Theorem \ref{main-thm1}, there is $K_4>0$ such that
\begin{align}
\label{new-thm2-eq5}
& \int_{s}^{t}\| A^{\beta}e^{-A(t-\lambda)} \big((1+a(\lambda,\cdot))u(\lambda,\cdot;s,u_0)- b(\lambda) u^2(\lambda, \cdot;s,u_0)\big)\|_{L^{q/2}(\Omega)}d\lambda\nonumber \\
     & \le   \int_s^{t} (t-\lambda)^{-\beta}e^{-\gamma(t-\lambda)}   \big(|\Omega|^{2/q}+a_{\max} \|u(\lambda,\cdot;s u_0)\|_{L^{q/2}}+b_{\max}\|u^2(\lambda,\cdot;s, u_0)\|_{L^{q/2}}\big)  d\lambda\nonumber \\
    &\leq K_4 \int_{s}^{t} (t-\lambda)^{-\beta} e^{-\gamma(t-\lambda)} d\lambda.
\end{align}
By \eqref{new-thm2-eq1}-\eqref{new-thm2-eq5}, for any $\tau>0$ and ${ 0<\theta<1-\frac{2N}{q}}$, there is
$M_3^*(\theta,\tau)>0$ such that  \eqref{new-set-eq1} holds. (1) is thus proved.

\smallskip

(2) First of all, it is not difficult to see that \eqref{new-set-eq2} follows directly  from Lemma \ref{main-lem0} and Theorem \ref{main-thm1}. It then suffices to prove \eqref{new-set-eq3}.

For any $u_0$ satisfying ${ (1.2)'}$, by \eqref{new-set-eq2}, there is $T_0(u_0)>0$ such that for any $s\in\RR$,
$$
\int_\Omega u^{-p}(t,x;s,u_0)dx\le 2M_1^*\quad {\rm and}\quad \int_\Omega u^q(t,x;s,u_0)dx\le 2M_2^*\quad \forall\,
t\ge s+T_0(u_0).
$$
Observe that
\begin{equation}
\label{new-thm2-eq6-1}
(u(t,x;s,u_0),v(t,x;s,u_0))=(u(t,x;\tilde s,\tilde u_0),
v(t,x;\tilde s,\tilde u_0))\quad \forall\, t\ge \tilde s,
\end{equation}
where $\tilde s=s+T_0(u_0)$ and $\tilde u_0(x)=u(s+T_0(u_0),x;s,u_0)$. Hence
$$
\int_\Omega u^{-p}(t,x;\tilde s,\tilde u_0)dx\le 2M_1^*\quad {\rm and}\quad \int_\Omega u^q(t,x;\tilde s,\tilde u_0)dx\le 2M_2^*\quad \forall\,
t\ge \tilde s.
$$
By the arguments in the proof of \eqref{new-set-eq1},   there is $M_4^*(\theta)$ such that
$$
\|u(t,\cdot;\tilde s,\tilde u_0)\|_{C^\theta(\bar\Omega)}\le M_4^*(\theta)\quad \forall\, t\ge \tilde s+1.
$$
This together with \eqref{new-thm2-eq6-1} implies that \eqref{new-set-eq3} holds with $T(u_0)=T_0(u_0)+1$.
\end{proof}

We now
 prove Theorem \ref{main-thm3}. In the rest of this section, to indicate the dependence of $(u(t,x;s,u_0),v(t,x;s,u_0))$
on $a(t,x), b(t,x)$, we put
$$
(u(t,x;s,u_0,a,b),v(t,x;s,u_0,a,b))=(u(t,x;s,u_0),v(t,x;s,u_0)).
$$

\begin{proof}
[Proof of Theorem \ref{main-thm3}]
We prove this theorem by contradiction.  Assume that there is no $m^*>0$ such that \eqref{new-set-eq6} holds.
Then for any $n\in\NN$,  there is $u_n$ satisfying ${(1.2)'}$ such that
\begin{equation}
\label{new-set-eq7}
\liminf_{t-s\to\infty} \inf_{x\in\Omega}u(t,x;s,u_n,a,b)\le m_n:=\frac{1}{n}.
\end{equation}

Let $p,q, \theta, M_1^*,M_2^*$, and  $M_4^*(\theta)$ be as in Theorem \ref{main-thm2}.  By Theorem \ref{main-thm2} and
Lemma \ref{prelim-lm-2}, there is $T_n>0$ such that
\begin{equation}
\label{new-set-eq8}
\|u(t,\cdot;s,u_n,a,b)\|_{C^\theta(\bar\Omega)}\le M_4^*(\theta)\quad \forall \, t-s\ge T_n
\end{equation}
and
\begin{equation}
\label{new-set-eq9}
\int_\Omega u(t,x;s,u_n,a,b)dx\ge \frac{\tilde M_1^*}{2}\quad \forall\, t-s\ge T_n.
\end{equation}
By \eqref{new-set-eq7}, there are $t_n,s_n\in\RR$ with $t_n-s_n\ge T_n+1$ and $x_n\in\bar\Omega$ such that
\begin{equation}
\label{new-set-eq10}
u(t_n,x_n;s_n,u_n,a,b)\le \frac{2}{n}.
\end{equation}

Let
$$
a_n(t,x)=a(t+t_n-1,x),\quad b_n(t,x)=b(t+t_n-1,x).
$$
Observe that
\begin{align}
\label{new-set-eq11}
u(t_n,x;s,u_n,a,b)&=u(t_n,x;t_n-1,u(t_n-1,\cdot;s_n,u_n,a,b),a,b)\nonumber\\
&=u(1,x;0,u(t_n-1,\cdot;s_n,u_n,a,b),a_n,b_n).
\end{align}
By \eqref{new-set-eq8} and {\bf (H)}, without loss of generality, we may assume that there are $u_0^*,u_1^*\in C(\bar\Omega)$ and $a^*(t,x),b^*(t,x)$ satisfying {\bf (H)} such that
$$
\lim_{n\to\infty} u(t_n-1,x;s_n,u_n,a,b)=u_0^*(x),
\quad \lim_{n\to\infty} u(t_n,x;s_n,u_n,a,b)=u_1^*(x)
$$
uniformly in $x\in\bar\Omega$, and
$$
\lim_{n\to\infty} a_n(t,x)=a^*(t,x),\quad \lim_{n\to\infty} b_n(t,x)=b^*(t,x)
$$
uniformly in $x\in\bar\Omega$ and locally uniformly in $t\in\RR$. It then follows from Lemma \ref{prelim-lm-011} and \eqref{new-set-eq11} that
\begin{equation}
\label{new-set-eq12}
u_1^*(x)=u(1,x;0,u_0^*,a^*,b^*).
\end{equation}

By \eqref{new-set-eq8}, \eqref{new-set-eq9}, and the Dominated Convergence Theorem,
$$
\int_\Omega u_0^*(x)=\lim_{n\to\infty}\int_\Omega u(t_n-1,x;s_n,u_n,a,b)dx\ge \frac{\tilde M_1^*}{2}.
$$
Then by \eqref{new-set-eq12} and {the strong maximum principle} for parabolic equations,
$$
\inf_{x\in\bar\Omega} u_1^*(x)>0.
$$
This together with \eqref{new-set-eq10}  implies that
$$
\frac{2}{n}\ge \inf_{x\in \Omega} u(t_n,x;s_n,u_n,a,b)\ge \frac{1}{2}\inf_{x\in\Omega} u_1^*(x)\quad \forall\, n\gg 1,
$$
which is a contradiction.
Therefore, there is $m^*>0$ such that for any $u_0$ satisfying \eqref{initial-cond-eq},
$$
\liminf_{t-s\to\infty}\inf_{x\in\Omega}u(t,x;s,u_0,a,b)\ge m^*.
$$
The theorem is thus proved.
\end{proof}

\section{Existence of  positive entire solutions}

In this section, we study the existence of positive entire solutions and prove
Theorem \ref{main-thm4}.

We first prove a lemma.

\begin{lemma}
\label{proof-bdd-lm1}  Let  $p>0$,  $q>2N$,  and $M_1^*,M_2^*>0$ be as in Theorem \ref{main-thm2}.
 The set
\begin{equation*}
\mathcal{E}=\left\{u\in C^0(\bar \Omega)\,|\, u\ge 0,\,\, \int_\Omega u(x)\le M_0^*,\,\, \int_\Omega u^{-p}(x)dx\le M_1^*,\,\, \int_\Omega u^{q}(x)\le M_2^*\right\}
\end{equation*}
is a convex, closed subset of $C^0(\bar\Omega)$.
\end{lemma}

\begin{proof}
First, we prove the set $\mathcal{E} $ is convex. {Note that $g(u)=u^r$   is a convex function on $(0,\infty)$, where $r=1$, $-p$ or $q$.   Hence, for  any $ \lambda \in [0,1]$ and $ u_*,u^* >0$, we have
\begin{equation}
\label{eq-convx-2}
    g(\lambda u_* + (1-\lambda) u^*) \leq \lambda g(u_*) + (1-\lambda) g(u^*).
\end{equation}
For any $u_1, u_2 \in \mathcal{E},$  we have $\int_{\Omega} u_1^{-p}(x)dx\leq M_1^*$, $ \int_{\Omega} u_2{^{-p}}(x)dx \leq M_1^*$, and $u_1(x)>0, u_2(x)>0$ for a.e. $x\in\Omega$.    By \eqref{eq-convx-2}, we have that
$$
g(\lambda u_1(x)+(1-\lambda)u_2(x))\le \lambda g( u_1(x))+(1-\lambda)g(u_2(x))\quad {\rm for}\quad a.\, e. \,\,\, x\in\Omega.
$$
This implies that
\begin{align*}
    \int_{\Omega}  g( \lambda u_1(x) + (1-\lambda) u_2(x))  dx  &\leq\lambda \int_{\Omega} g(u_1(x))dx + (1-\lambda) \int_{\Omega} g(u_2(x))dx.
\end{align*}
This implies that  for any $\lambda\in [0,1]$ and $u_1,u_2\in\mathcal{E}$,
 $\lambda u_1+ (1-\lambda) u_2 \in \mathcal{E}.$ Thus,  the set $\mathcal{E} $ is convex.}

Next, we prove the set $\mathcal{E} $ is closed. Suppose that $u_n \in \mathcal{E}$, $u\in C^0(\bar\Omega)$, and
  $u_n \rightarrow u$ in $C^{0}(\bar \Omega)$. Then $u(x)\ge 0$ for any $x\in\bar\Omega$, and  for any $\varepsilon>0$,  there is  $N_\varepsilon$ such that for any $n \ge N_\varepsilon$, there holds
 $$|u_n(x)-u(x)| < \varepsilon\quad \forall\, x \in \bar \Omega.
$$
 This implies that
 \begin{equation}
\label{new-enq-1}
0\le u_n(x)=u_n(x)-u(x)+u(x) \leq \varepsilon + u(x)\quad \forall\, n\ge N_\varepsilon,\,\, x\in\bar\Omega.
\end{equation}

For any $\varepsilon>0$, let
$$
\Omega_\varepsilon=\{x\in\Omega\,|\, u(x)\le \varepsilon\}.
$$
Let
$$
\Omega_0=\{x\in\Omega\,|\, u(x)=0\}.
$$
By \eqref{new-enq-1},
$$
0\le u_n(x)\le 2\varepsilon\quad \forall\, x\in\Omega_\varepsilon.
$$
Since $\int_\Omega u_n^{-p}(x)dx\le M_1^*$, we have that
\begin{equation*}
    |\Omega_{\epsilon}|=\int_{\Omega} u_n^{-p} \cdot u_n^{p}\le \int_\Omega u_n^{-p} \cdot (2\epsilon)^p\le (2\epsilon)^p \int_\Omega u_n^{-p} \le (2\epsilon)^p M_1^*.
\end{equation*}
This implies that
$$
|\Omega_0|=0.
$$
{Let $g(u)$ be as in the above.
Then by Fatou's lemma, we have
\begin{align*}
  \int_\Omega g(u(x))dx=  \int_{\Omega\setminus\Omega_0} g(u(x))dx= \int_{\Omega\setminus \Omega_0} \lim_{n \rightarrow \infty} g( u_n(x))dx \leq \lim_{n \rightarrow \infty} \inf \int_{\Omega\setminus\Omega_0} g(u_n(x))dx.
\end{align*}
This implies that
$$
\int_\Omega u(x)dx\le M_0^*,\,\, \int_\Omega u^{-p}(x)dx\le M_1^*,\,\, \int_\Omega u^q(x)dx\le {M_2^*}.
$$
Hence $u\in\mathcal{E}$  and the set $\mathcal{E} $ is closed.}
\end{proof}

We now prove Theorem \ref{main-thm4}.

\begin{proof} [Proof of Theorem \ref{main-thm4}]

(1)
First of all, fix $u_0\in \mathcal{E},$ where $\mathcal{E}$ as in \eqref{set-E-eq}.  By Theorem \ref{main-thm2},  $u(t,\cdot;s,u_0)\in\mathcal{E}$ for any $t>s$.

Next, for any $n\in \NN$,  define $u_n(x)=u(0,x,;-n,u_0)$. Fix $\theta\in (0,\frac{2N}{q})$. By Theorem \ref{main-thm2} again, there is
$\tilde M_3^*>0$ such that
\begin{equation}
\label{main-thm3-eq1}
      \|u(t,\cdot;-n,u_n)\|_{C^{\theta}} =\sup_{x \in \Omega}{|u(t,x;-n,u_0)|} + \sup{\frac{|u(t,x;-n,u_0)-u(t,y;-n,u_0)|}{|x-y|^{\theta}}} \le \tilde M_3^*
\end{equation}
for all $t\ge -n+1$ and $n\ge 1$.
This implies that the sequence  $\{u_n(\cdot)\}$ is uniformly bounded and equicontinuous on $\bar \Omega$. By the Arzel\'a-Ascoli Theorem and Lemma \ref{proof-bdd-lm1}, there are $n_k$, $u^*_0 \in \mathcal{E}$ such that $u_{n_k} \rightarrow u^*_0$ in $C^0(\bar\Omega)$ as $n_k \to \infty.$

Note that \vspace{-0.1in}
\begin{equation*}
    u(\cdot,t;-n_k,u_{0})=u(\cdot,t;0,u(\cdot,0;-n_k,u_0))=u(\cdot,t;0,u_{n_k})\quad \forall \, t>0.
\end{equation*}
By Lemma  \ref{prelim-lm-011},  for any $t>0$,
\begin{equation*}
    \lim_{n_k \to \infty}  u(\cdot,t;-n_k,u_{0})=\lim_{n_k \to \infty} u(\cdot,t;0,u_{n_k})= u(\cdot,t;0,u^*_0) \quad {\rm in}\quad  C^0(\bar\Omega).
\end{equation*}
Since $v(\cdot,t;0,u^*_0)=A^{-1}u(\cdot,t;0,u^*_0),$ we also have that for any $t>0$,
\begin{equation*}
    \lim_{n_k \to \infty}  v(\cdot,t;-n_k,u_{0})=\lim_{n_k \to \infty}  A^{-1} u(\cdot,t;-n_k,u_{0}) = v(\cdot,t;0,u^*_0) \quad {\rm in}\quad C^0(\bar\Omega).
\end{equation*}

We now prove that $u(\cdot,t;0,u^*_0)$ has a backward extension on $(-\infty,0)$. By \eqref{main-thm3-eq1} and the Arzel\'a-Ascoli Theorem, without loss of generality, we may assume that for any $m\in\NN$, there is $u_m^*(\cdot)\in\mathcal{E}$ such that
$u(\cdot,-m;-n_k,u_0)\to  u_m^*(\cdot)$
in $C^0(\bar \Omega)$.
By Lemma  \ref{prelim-lm-011} again,
\begin{equation*}
    \lim_{n_k \to \infty} u(\cdot,t;-n_k,u_0)=\lim_{n_k \to \infty} u(\cdot,t;-m,u(\cdot,-m;-n_k,u_0)) = u(\cdot,t;-m, u_m^*)
\end{equation*}
and for $t>-m$. Note that $u_0^*=u(\cdot,0;-m,u_m^*)$. This implies that $u^*(x,t;0,u_0^*)$ has a backward extension up to $t=-m$. Letting $m\to \infty$ yields that $u^*(x,t)$ has a backward extension on $(-\infty,0)$.

Finally, let us denote $u^*(t,x)=u^*(t,x;0,u_0^*),$ hence $v^*(t,x)=A^{-1} u^*(t,x).$ Then $(u^*(t,x),v^*(t,x))$ is an entire solution of  \eqref{main-eq} and $u^*(t,\cdot)\in\mathcal{E}$ for any $t\in\RR$.  (1) is thus proved.

\medskip

(2)
Assume that there exists $T>0$ such that $a(t+T,x)=a(t,x)$ and $b(t+T,x)=b(t,x)$. First,  define the map $\mathcal{T} (T):\mathcal{E} \rightarrow C^0(\bar\Omega)$ as follows:
\begin{equation*}
    \mathcal{T} (T)u_0=u( T,\cdot; 0,u_0)\quad \forall \, u_0\in\mathcal{E}.
\end{equation*}
By Theorem \ref{main-thm2}, $\mathcal{T}(T)$ maps $\mathcal{E}$ into $\mathcal{E}$. By Lemma \ref{prelim-lm-01} and Theorem \ref{main-thm2},
$\mathcal{T}(T):\mathcal{E}\to\mathcal{E}$ is continuous and relatively compact in $C^0(\bar\Omega)$.
By Lemma  \ref{proof-bdd-lm1},  $\mathcal{E}$ is a convex closed subset of $C^0(\bar\Omega)$.
Then,  by Schauder fixed point theorem (see Lemma \ref{sch-fxd-point}), one can find $u^T \in \mathcal{E} $ such that $\mathcal{T} (T)u^T=u^T,$ that is, $u(\cdot,T;s,u^T)=u^T(\cdot).$
Thus,
\begin{equation*}
    u(t+T,\cdot;0,u^T)=u(t,\cdot;T,u(T,\cdot;0,u^T))=u(t,\cdot;0,u^T).
\end{equation*}
Hence, $u(t,\cdot;0,u^T)$ is periodic with period $T$. Moreover, the second equation of (1.3) and the uniqueness of solutions of
\begin{equation*}
\begin{cases}
-\Delta v +v = u(t,\cdot;0,u^T), & x\in \Omega \cr
\frac{\p v}{\p n}=0,\quad x\in\p\Omega,
\end{cases}
\end{equation*}
it is obtained that $v(t,\cdot;0,u^T)=(I-\Delta)^{-1}u(t,\cdot;0,u^T)$ is periodic with period $T,$ which concludes that $(u(t,\cdot;0,u^T),v(t,\cdot;0,u^T))$ is a positive periodic solution of (1.3) with $u(t,\cdot;0,u^T)\in\mathcal{E}$ for any $t\in\RR$.. (2) is thus proved.

\medskip

(3) Assume that $a(t,x)=a(x)$ and $b(t,x)=b(x).$ Observe that every $T>0$ is a period for $a(t,x)$ and $b(t,x).$
By  (2),
 for fixed $T>0,$ there exists $u^{T} \in \mathcal{E}$ such that $(u(t,\cdot;0,u^T),(I-\Delta)^{-1}u(t,\cdot;0,u^T))$ is a positive  periodic solution of \eqref{main-eq} with period $T$.

Let $T_n=\frac{1}{n}$ and $u_n(x)=u^{T_n}(x)$. Then $u_n(x)=u(kT_n,x;0,u_n)$ for all $k\in\NN$.
By Theorem \ref{main-thm2}, $\{u_n(\cdot)\}$ is a uniformly bounded and equi-continuous sequence in $C^0(\bar\Omega)$.
Then by the  Arzel\'a-Ascoli theorem,
there exist a subsequence $\{u_{n_k}\}$ and  $u^* \in \mathcal{E}$ such that $u_{n_k} \to u^*$ in $C^0(\bar\Omega)$ as $n_k \to \infty.$

We claim that
\begin{equation}
\label{entire-eq7}
u(\cdot,t;0,u^*)=u^*(\cdot)\quad \text{for all}\,\, t\ge 0.
\end{equation}
For any $t>0$ and $\epsilon>0$, let $\tau_{n_k}\in [0,T_{n_k})$ be such that
$$
\frac{t-\tau_{n_k}}{T_{n_k}}\in \ZZ^+.
$$
Then
$$
\tau_{n_k}\to 0\quad {\rm and}\quad k\to\infty.
$$
By Lemma \ref{prelim-lm-01},
$$
 \|u(\tau_{n_k},\cdot;0,u^*)-u^*(\cdot)\|_\infty<\epsilon\quad \forall\, k\gg 1.
$$
By Lemma \ref{prelim-lm-011},
$$
\|u(\tau,\cdot;0,u^*)-u(\tau,\cdot;0,u_{n_k})\|_\infty< \epsilon\quad  \forall\,  \tau\in [0,t] \quad {\rm and}\quad \forall\,  k\gg 1.
$$
It then follows that
\vspace{-0,1in}\begin{align*}
|u(t,x;0,u^*)-u^*(x)|&\le |u(t,x;0,u^*)-u(t,x;0,u_{n_k})| +|u(t,x;0,u_{n_k})-u^*(x)|\quad (\text{choose}\,\, k\gg 1)\\
&= |u(t,x;0,u^*)-u(t,x;0,u_{n_k})| +|u(\tau_{n_k},x;0,u_{n_k})-u^*(x)|\\
& \le |u(t,x;0,u^*)-u(t,x;0,u_{n_k})| +|u(\tau_{n_k},x;0,u_{n_k})- u(\tau_{n_k},x;0,u^*)|\\
&\quad +|u(\tau_{n_k},x;0,u^*)-u^*(x)|\\
&\le 3\epsilon\quad \forall\, x\in\bar\Omega.
\end{align*}
Letting $\epsilon\to 0$ entails \eqref{entire-eq7}. This shows that the solution pair $(u(\cdot,t;0,u^*),v(\cdot,t;0,u^*))$ with $v(\cdot,t;0,u^*)=(I-\Delta)^{-1}u(\cdot,t;0,u^*)$ is a steady state solution of \eqref{main-eq} with $u^*(\cdot)\in\mathcal{E}$. (3) is thus proved.
\end{proof}

\end{document}